\documentclass[12pt]{article}
\usepackage{amsmath, amsfonts, latexsym, amssymb, graphicx, mathtools,  wasysym, enumitem, pifont}

\usepackage[all]{xy}
\usepackage{hyperref}

\hypersetup{colorlinks,citecolor=blue,linkcolor=blue}

\title{Near Ascending HNN-Extensions and a Combination Result for Semistability at Infinity}
\author{Michael Mihalik}
\usepackage{tikz}
\newtheorem{theorem}{Theorem}[section]

\newtheorem{lemma}[theorem]{Lemma}

\newtheorem{remark}[theorem]{Remark}

\newtheorem{example}[theorem]{Example}

\newcounter{claimnum}

\newcounter{definitionnum}

\newenvironment{proof}{\addvspace{12pt}\noindent{\bf Proof:}}{
$\Box$\par\addvspace{12pt}}

\date{\today}
\begin{document}
\maketitle

\begin{abstract}
Semistability at $\infty$ is an asymptotic property of finitely presented groups that is needed in order to effectively define the fundamental group at $\infty$ for a 1-ended group. It is an open problem whether or not all finitely presented groups have semistable fundamental group at $\infty$. While many classes of groups are known to contain only semistable at $\infty$ groups, there are only a few combination results for such groups. Our main theorem is such a result. 
 
 \medskip
 
\noindent {\bf Main Theorem.} Suppose $G$ is the fundamental group of a connected reduced graph of groups, where each edge group is infinite and finitely generated, and each vertex group is finitely presented and either 1-ended and semistable at $\infty$ or has an edge group of finite index. Then $G$ is 1-ended and semistable at $\infty$.

An important part of the proof of this result is the semistability part of the following:

\medskip

\noindent{\bf Theorem.} Suppose $H_0$ is an infinite finitely presented group, $H_1$ is a subgroup of finite index in $H_0$, $\phi:H_1\to H_0$ is a monomorphism and $G=H_0\ast_\phi$ is the resulting HNN extension. Then $G$ is 1-ended and semistable at $\infty$. If additionally, $H_0$ is 1-ended, then $G$ is simply connected at $\infty$. 
\end{abstract}

\section{Introduction}

A connected locally finite complex $K$ has {\it semistable fundamental group at} $\infty$ if any two proper rays converging to the same end of $K$ are properly homotopic. If for any compact $C\subset K$ there is a compact $D\subset K$ such that loops in $K-D$ are homotopically trivial in $K-C$, then $K$ is {\it simply connected at $\infty$}. If $G$ is a finitely presented group then $G$ has semistable (respectively simply connected) fundamental group at $\infty$ if for some (equivalently any) finite complex $X$ with $\pi_1(X)=G$, the universal cover of $X$ has semistable (respectively simply connected) fundamental group at $\infty$. Both semistability and simple connectivity at $\infty$ are quasi-isometry invariants of finitely presented groups (\cite{G} and \cite{BR93}).  If a finitely presented group $G$ has semistable fundamental group at $\infty$ (respectively is simply connected at $\infty$) then $H^2(G,\mathbb ZG)$ is free abelian (respectively trivial) (See Theorems 16.5.1 and 16.5.2, \cite{G}). A open question (apparently due to H. Hopf) ask if $H^2(G,\mathbb ZG)$ is free abelian for all finitely presented groups $G$. There are many classes of groups, all of whose (finitely presented) members are know to have semistable fundamental group at $\infty$. These include word hyperbolic groups (combining work of B. Bowditch \cite{Bow99B},  G. Levitt \cite{Lev98}, G. Swarup \cite{Swarup} and M. Bestvina and G. Mess \cite{BM91}), CAT(0) cube groups (S. Shepherd \cite{SSh}), Artin and Coxeter groups (Theorems 1.1 and 1.4, \cite{M96}), many relatively hyperbolic groups ((Theorem 1.1, \cite{MS18}) and (Theorem 1.5, \cite{HM20})) and groups containing a finitely generated infinite subcommensurated subgroup of infinite index (Theorem 1.9, \cite{M6}). There are only a few semistability combination results. Perhaps the most significant one states:

\begin{theorem}\label{MTComb}(\cite{MT1992}) {\bf (MTComb)} If $G$ is the fundamental group of a finite graph of groups where each vertex group is finitely presented with semistable fundamental group at $\infty$, and each edge group is finitely generated, then $G$ has semistable fundamental group at $\infty$. 
\end{theorem}

Notice that finitely generated free groups trivially have semistable fundamental group at $\infty$. Exotic 1-ended groups can be written as $G_1\ast_{G_3} G_2$ for $G_i$ a finitely generated free group. These groups have semistable fundamental group at $\infty$ by Theorem \ref{MTComb}. 

\begin{example} \label{Examples3} {\bf (Examples3)} Let $F_i$ be the free group of rank $i>0$. Three groups $\Lambda_1$, $\Lambda_2$ and $\Lambda_3$ are constructed in \cite{Ratt07}.  Each group $\Lambda_i$ can be decomposed in two ways as amalgamated products $F_9\ast_{F_{81}} F_9$, such that $F_{81}$ has index 10 in both $F_9$ factors. The group $\Lambda_1$ is a finitely presented, torsion free simple group. The groups $\Lambda_2$ and $\Lambda_3$ are not simple, but $\Lambda_2$ is virtually simple and $\Lambda_3$ has no non-trivial finite quotients.
\end{example}

Since the $F_{81}$ edge group has finite index in each vertex group, it is commensurated in each $\Lambda_i$. This gives a second way of seeing the $\Lambda_i$ have semistable fundamental group at $\infty$ (see Theorem \ref{MainCM}). Our main theorem generalizes this phenomenon. 
\begin{theorem} \label{SSDecomp} {\bf(SSDecomp)}
Suppose $G$ is the fundamental group of a connected reduced graph of groups, where each edge group is infinite and finitely generated, and each vertex group is finitely presented and either 1-ended and semistable at $\infty$ or has an edge group of finite index. 
Then $G$ is 1-ended and semistable at $\infty$. 
\end{theorem}

First a clarification. If an edge $e$ is a loop at a vertex $v$ in our graph, then the group $G_v$ has two subgroups $G_1$ and $G_2$, both of which are isomorphic to $G_e$. There is an isomorphism $\phi:G_1\to G_2$ determining a subgroup of $G$ that is an HNN-extension of $G_v$. If either $G_1$ or $G_2$ has finite index in $G_v$, we say $G_e$ has finite index in $G_v$.  

In the setting of this theorem, even when the graph is an edge path loop (with at least 2 vertices) and each edge group has finite index in its initial vertex, it seems that one cannot use commensurability ideas in order to prove semistability. Instead, for such a loop, the intersection of the edge groups will have finite index in the first vertex group producing an important  {\it near ascending HNN-extension} subgroup as considered in the following result.  

\begin{theorem} \label{MMFI} {\bf (MMFI)} Suppose $H_0$ is an infinite finitely presented group, $H_1$ is a subgroup of finite index in $H_0$, $\phi:H_1\to H_0$ is a monomorphism and $G=H_0\ast_\phi$ is the resulting HNN extension. Then $G$ is 1-ended and semistable at $\infty$. If additionally, $H_0$ is 1-ended, then $G$ is simply connected at $\infty$. 
\end{theorem}

If $H_0=H_1$ then $G$ is an {\it ascending HNN-extension} with base group $H_0$. This case of Theorem \ref{MMFI} was proved as (Theorem 3.1,  \cite{HNN1}), but is insufficient for our purposes. The proof of Theorem \ref{MMFI} is fundamentally  more difficult than that of (Theorem 3.1, \cite{HNN1}). This theorem is a critical piece of our proof of the Main Theorem. 

Semistability preliminaries are covered in \S \ref{SS}. We prove that a near ascending HNN-extension of an infinite finitely generated group is 1-ended in \S \ref{NHH1E}. The proof of  Theorem \ref{MMFI} is presented in \S \ref{PFMMFI} and the proof of the Main Theorem is concluded in \S \ref{PMT}.

\section{Semistability and End Preliminaries {\bf (SS)}} \label{SS}

While semistability makes sense for multiple ended spaces, we are only interested in 1-ended spaces in this article. Suppose $K$ is a 
locally finite connected CW complex. A {\it ray} in $K$ is a continuous map $r:[0,\infty)\to K$. A continuous map $f:X\to Y$ is {\it proper} if for each compact subset $C$ of $Y$, $f^{-1}(C)$ is compact in $X$. Proper rays $r,s:[0,\infty)\to K$ {\it converge to the same end} if for any compact set $C$ in $K$, there is an integer $k(C)$ such that $r([k,\infty))$ and $s([k,\infty))$ belong to the same component of $K-C$. 
The space $K$ has {\it semistable fundamental group at $\infty$} if any two proper rays $r,s:[0,\infty)\to K$ that converge to the same end
are properly homotopic  (there is a proper map $H:[0,1]\times [0,\infty)\to X$ such that $H(0,t)=r(t)$ and $H(1,t)=s(t)$). Note that when $K$ is 1-ended, this means that $K$ has semistable fundamental group at $\infty$ if any two proper rays in $K$ are properly homotopic. 
Suppose  $C_0, C_1,\ldots $ is a collection of compact subsets of a locally finite 1-ended complex $K$ such that $C_i$ is a subset of the interior of $C_{i+1}$ and $\cup_{i=0}^\infty C_i=K$, and $r:[0,\infty)\to K$ is proper, then $\pi_1^\infty (K,r)$ is the inverse limit of the inverse system of groups:
$$\pi_1(K-C_0,r)\leftarrow \pi_1(K-C_1,r)\leftarrow \cdots$$
This inverse system is pro-isomorphic to an inverse system of groups with epimorphic bonding maps if and only if $K$ has semistable fundamental group at $\infty$ (see Theorem 2.1 of \cite{M1} or Theorem 16.1.2 of \cite{G}).

(Theorem 2.1, \cite{M1}) and (Lemma 9, \cite{M86}) promote several equivalent notions of semistability. We will use the second formulation in the next result to prove Theorem \ref{MMFI} and the third formulation in the proof of Theorem \ref{SSDecomp}.
\begin{theorem}\label{ssequiv} {\bf (ssequiv)}
Suppose $K$ is a 1-ended, locally finite and connected CW-complex. Then the following are equivalent:
\begin{enumerate}
\item $K$ has  semistable fundamental group at $\infty$.
\item For some (equivalently any) proper ray $r:[0,\infty )\to K$ and compact set $C$, there is a compact set $D$ such that for any third compact set $E$ and loop $\alpha$ based on $r$ and with image in $K-D$, $\alpha$ is homotopic rel$\{r\}$ to a loop in $K-E$, by a homotopy with image in $K-C$. 
\item For any compact set $C$ there is a compact set $D$ such that if $r$ and $s$ are rays based at $v$ and with image in $K-D$, then $r$ and $s$ are properly homotopic rel$\{v\}$, by a proper homotopy in $K-C$. 

\medskip

\hbox{If in addition $K$ is simply connected, then:}

\item If $r$ and $s$ are rays based at $v$ then
$r$ and $s$ are properly homotopic $rel\{v\}$.
\end{enumerate}
\end{theorem}

In a connected locally finite CW-complex $K$, every proper ray is properly homotopic to a proper edge path ray and every loop at a vertex of $K$ is homotopic (by a small homotopy) to an edge path loop. 

A subgroup $H$ of a group $G$ is {\it commensurated} in $G$ if for each $g\in G$, the subgroup $H\cap (g^{-1}Hg)$ has finite index in both $H$ and $g^{-1}Hg$. The following two results will be used in our proof of Theorem \ref{SSDecomp}.
There is a notion of semistability at $\infty$ for finitely generated groups that agrees with our definitions when the group is finitely presented. We mention this fact since it comes up in the following two results.

\begin{theorem}\label{MainCM} (\cite{CM2}) {\bf (MainCM)}
If a finitely generated group $G$ has an infinite, finitely generated, commensurated subgroup $Q$, and $Q$ has infinite index in $G$, then $G$ is 1-ended and semistable at $\infty$.
\end{theorem}

\begin{theorem}\label{MComb} (Theorem 3, \cite{M4}) {\bf (MComb)} 
Suppose  $A$ and $B$ are finitely generated 1 or 2-ended subgroups of the finitely generated group $G$ and both $A$ and $B$ are semistable at $\infty$.  If the set $A\cup B$ generates $G$ and the group $A\cap B$ contains a finitely generated infinite subgroup then $G$ has semistable fundamental group at $\infty$. 
\end{theorem}

The conditions of the following result are the most general we know in order to define ends of a space. In this article we are only interested in locally finite CW-complexes. 

\begin{theorem} \label{ES} {\bf(ES)} Suppose $X$ is connected, locally compact, locally connected
and Hausdorff and $C$ is compact in $X$,
then $C$ union all bounded components of $(X - C)$ is compact, and $X-C$ has only finitely many unbounded components.
Here bounded means compact closure.
\end{theorem}

If $X$ satisfies hypotheses of Theorem \ref{ES}, then the {\it number of ends} of $X$ is the largest number of unbounded components one can obtain in the compliment of a compact subset of $X$. If this number is unbounded, then $X$ has infinitely many ends. 
If $K$ is a locally finite, connected CW complex, then $K$ has 1-end if and only if any two proper rays in $K$ converge to the same end. Note that if two proper rays are properly homotopic, then they converge to the same end. 

\section {Near Ascending HNN-Extensions are 1-ended (NHH1E)}\label{NHH1E}
If $A$ is a group with finite generating set $\mathcal A$ then let $\Gamma(A,\mathcal A)$ be the Cayley graph of $A$ with respect to $\mathcal A$. This means that the vertex set of $\Gamma (A,\mathcal A)$ is $A$ and there is an edge (labeled $a$) between the vertex $v$ and $w$ if $va=w$ for some $a\in \mathcal A$.
The following result is somewhat standard. We include it for completeness.
\begin{theorem} \label{MMFI1} {\bf (MMFI1)} Suppose $H_0$ is an infinite finitely generated group, $H_1$ is a subgroup of finite index in $H_0$, $\phi:H_1\to H_0$ is a monomorphism and $G=H_0\ast_\phi$ is the resulting HNN extension. Then $G$ is 1-ended.
\end{theorem}
\begin{proof}
Choose a finite generating set $\mathcal H_0$ for $H_0$ where $\mathcal H_1\subset \mathcal H_0$ generates $H_1$. Then $\mathcal H_0$ along with $t$ (the stable letter) generates $G$. 
Consider: 
$$\Gamma_1=\Gamma(H_1,\mathcal H_1)\subset \Gamma_0=\Gamma(H_0,\mathcal H_0)\subset \Gamma=\Gamma(G,\{t\}\cup \mathcal H_0).$$
If $v$ (an element of $G$) is a vertex of $\Gamma$, then $v\Gamma_0$ and $v\Gamma_1$ are translates of $\Gamma_0$ and $\Gamma_1$ respectively. Note that if $w$ is a vertex of $v\Gamma_i$ ($i\in \{0,1\}$) then $v\Gamma_i=w\Gamma_i$. Our metric on $\Gamma$ is the edge path metric. If $v$ is a vertex of $\Gamma$ and $M>0$ an integer, then $B(v,M)$ is the ball of radius $M$ about $v$. Let $\ast$ be the identity vertex of $\Gamma$. There is a continuous map $P:\Gamma\to \mathbb R$ such that $P(\ast)=0$, each edge labeled $t$ is mapped to an interval $[n,n+1]$ for some $n\in \mathbb Z$ and each edge labeled by an element of $\mathcal H_0$ is mapped to an integer. This map extends the homomorphism $G\to \mathbb Z$ that kills the normal closure of $H_0$. 

We show that for any integer $N$ there is an integer $M>N$ such that any two vertices of $\Gamma-B(\ast, M)$ can be joined in $\Gamma-B(\ast,N)$. There are only finitely many $v\Gamma_0$ that intersect $B(\ast, N)$ non-trivially. Choose $M$ such that for each such $v\Gamma_0$, all bounded components of $v\Gamma_0-B(\ast,N)$ belong to $B(\ast,M)$ (see Theorem \ref{ES}). 
The next lemma follow easily from the fact that $A_1$ has finite index in $A_0$. 
\begin{lemma}\label{FIclose} {\bf (FIclose)}
Suppose $\mathcal A_1\subset \mathcal A_0$ are finite generating sets for the groups $A_1\subset A_0$ and that $A_1$ has finite index in $A_0$. Then there is an integer $N_{\ref{FIclose}}(A_0,\mathcal A_0,A_1, \mathcal A_1)$ such that if $v$ and $w$ are verteices of $\Gamma(A_0,\mathcal A_0)$  then there is an edge path in $\Gamma(A_0,\mathcal A_0)$ of length $\leq N_{\ref{FIclose}}(G_0,\mathcal A_0, G_1, \mathcal A_1)$ from $v$ to a vertex of $w\Gamma_1(A_1,\mathcal A_1)$.  
\end{lemma}

From this point on we let $N_{\ref{FIclose}}= N_{\ref{FIclose}} (H_0,\mathcal H_0, H_1,\mathcal H_1)$.
Suppose $e=(v,w)$ is an edge labeled $t$ and either $v$ or $w$ is in $B(\ast, M)$. Let $\alpha_v$ be an edge path at $v$ such that each edge of $\alpha$ is labeled by an element of $\mathcal H_1$ and such that the end point $a$ of $\alpha_v$ is in $\Gamma-B(\ast, M+1)$. Let $e_1=(a,b)$ be the edge at $a$ labeled $t$ and note that $e_1$ has image in $\Gamma-B(\ast, M)$. For each $h\in \mathcal H_1$, we have $t^{-1}ht=\phi(h)\in H_0$ and so there is an edge path $\beta_w$ in $w\Gamma_0$ from $w$ to $b$. Now suppose $x$ and $y$ are vertices of $\Gamma-B(\ast,M)$. We want to find a path in $\Gamma-B(\ast,N)$ connecting $x$ and $y$. Let $\tau$ be any edge path from $x$ to $y$ in $\Gamma$. If $e=(v,w)$ is an edge of $\tau$ labeled $t$, such that either $v$ or $w$ belongs to $B(\ast,M)$ then replace $e$ by the path $(\alpha_v, e_1, \beta_w^{-1})$ (above). In this way we produce an edge path $\tau_1$ from $x$ to $y$ such that each $t$ labeled edge of $\tau_1$ belongs to $\Gamma-B(\ast,M)$. Each maximal subpath of $\tau_1$ (none of whose edges is labeled $t$) has both end points in $\Gamma-B(\ast, M)$. Hence it suffices to show the end points of such paths can be joined by an edge path in $\Gamma-B(\ast,N)$. 

\begin{figure}
\vbox to 3in{\vspace {-2in} \hspace {.7in}
\hspace{-1 in}
\includegraphics[scale=1]{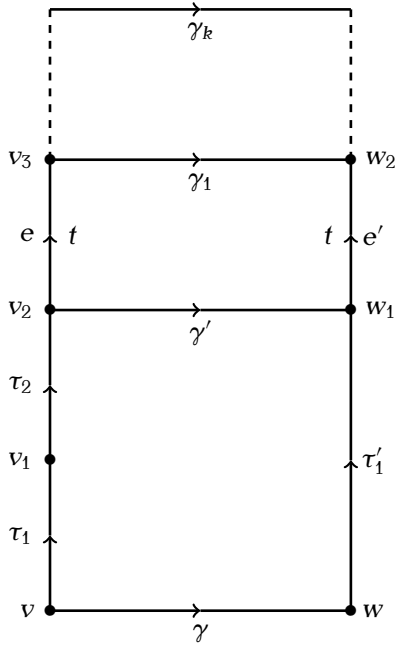}
\vss }
\vspace{-.1in}
\caption{One Endedness}
\label{Fig1}
\end{figure}

Recall that $P:\Gamma\to \mathbb R$ and $P(v\Gamma_0)=P(v)$. We have: 
$$P(B(\ast, N))=[-N,N].$$ 
Suppose $\gamma$ is an edge path with end points in $\Gamma-B(\ast,M)$, each edge of $\gamma$ is labeled by an element of $\mathcal H_0$ and $\gamma$ intersects $B(\ast,N)$, then $-N\leq P(\gamma)\leq N$. The strategy is to ``slide" $\gamma$ upward to a path $\gamma_k$ in level $N+1$ (which automatically avoids $B(\ast, N)$).  Assume the initial point of $\gamma$ is $v$ and the terminal point is $w$ (see Figure \ref{Fig1}).  By the definition of $M$,  the vertex $v$ is in an unbounded component of $v\Gamma_0-B(\ast,N)$. Hence there is an edge path $\tau_1$ in $v\Gamma_0-B(\ast,N)$ from $v$ to $v_1\in v\Gamma_0-B(\ast,M+N_{\ref{FIclose}}+1)$. Similarly choose $\tau'_1$ in $v\Gamma_0-B(\ast, N) $ from $w$ to $w_1\in v\Gamma_0-B(\ast, M+1)$. Note that: 
$$v\Gamma_0=v_1\Gamma_0=w\Gamma_0=w_1\Gamma_0.$$ 
 By the definition of $N_{\ref{FIclose}}$ we choose an edge path $\tau_2$ of length $\leq N_{\ref{FIclose}}$ in $w_1\Gamma_0-B(\ast, M+1)$ from $v_1$ to $v_2\in w_1\Gamma_1-B(\ast, M+1)$.  Let $e=(v_2,v_3)$ and $e'=(w_1,w_2)$ be edges labeled $t$. 
Let $\gamma'$ be an edge path in $w_1\Gamma_1$ from $v_2$ to $w_1$. The conjugation relations $t^{-1}ht=\phi(h)$ for $h\in \mathcal H_1$ (applied to each edge of $\gamma'$) give an edge path $\gamma_1$ in $v_3\Gamma_0$ from $v_3$ to $w_2$. 
Then:
 $$\{v_3,w_2\}\subset v_3\Gamma_0-B(\ast,M).$$
 Note that the paths $(\tau_1, \tau_2, e)$ (from $v$ to $v_3$) and $(\tau_1', e')$ (from $w$ to $w_2$) have image in $\Gamma-B(\ast, N)$. Hence we need only show there is an edge path in $\Gamma-B(\ast, N)$ from $v_3$ to $w_2$. Note that $P(\gamma_1)=P(\gamma)+1\geq -N+1$. Repeating this process at most $2N+1$ times produce a path from $v$ to $w$ in $\Gamma-B(\ast,N)$, and so $G$ is 1-ended. 
\end{proof}

\section{The Proof of Theorem \ref{MMFI} {\bf (PFMMFI)}} \label{PFMMFI}

This theorem generalizes both the Main Theorem and Theorem 3.1 of \cite{HNN1} (where $H_0=H_1$ and $G$ is an {\it ascending HNN-extension} of $H_0$).
Our proof of Theorem \ref{MMFI} has similarities to the one in \cite{HNN1}, but requires several new ideas. 

\begin{proof} Theorem \ref{MMFI1} implies that $G$ is 1-ended (even when $H_0$ is finitely generated and not finitely presented).  Again we consider the following finite presentation of $G$:
$$\mathcal P=\langle t, \mathcal H_0: \mathcal R, t^{-1}ht=\hat \phi(h)\hbox{ for all }h\in \mathcal H_1\rangle$$
where $\mathcal P_0=\langle \mathcal H_0:\mathcal R\rangle$ is a finite presentation for $H_0$. We assume  $\mathcal H_1\subset \mathcal H_0$ and $\mathcal P_1=\langle \mathcal H_1:\mathcal R_1\rangle$ (for some $\mathcal R_1\subset \mathcal R_0$)  is a finite presentation for $H_1$. Here $\hat \phi:F(\mathcal H_1)\to F(\mathcal H_0)$ is a homomorphism of free groups realizing $\phi$ and the relations $t^{-1}ht=\hat \phi(h)$ are called {\it conjugation relations}. Let 
$$q:F(\mathcal H_0\cup \{t\})\to G$$
be a quotient homomorphism with kernel equal to the normal closure of $\mathcal R$ union the set of conjugation relations.

\[
\xymatrix{
F(\mathcal H_1) \ar[r]^{\hat \phi} \ar[d]^q & F(\mathcal H_0) \ar[d]^q\\
H_1 \ar[r]^\phi & H_0 }
\]

Let  $\Gamma_1\subset \Gamma_0\subset \Gamma=\Gamma(\mathcal P)$ be the {\it Cayley 2-complexes} for $\mathcal P_1$, $\mathcal P_0$ and $\mathcal P$ respectively. The 1-skeleton of $\Gamma$ is the Cayley graph of $G$ with respect to the generating set $\mathcal H_0\cup \{t\}$. At each vertex there is a 2-cell attached according to each relation of $\mathcal P$. If the edges of an edge path in $\Gamma$ are labeled by $\mathcal S^{\pm1}$ for $\mathcal S\subset \mathcal H_0$, then we call it an $\mathcal S$ {\it path}. If $v(\in G)$ is a vertex of $\Gamma$ then $v\Gamma_1\subset v\Gamma_0\subset \Gamma$. Note that any edge path in $v\Gamma_0$ is an $\mathcal H_0$ path. There is a continuous map $p:\Gamma\to \mathbb R$ 
such that each vertex is mapped to an integer, a (directed) edge $e$ labeled $t$ is mapped linearly by $p$ to an interval $[n,n+1]$ (so that $p(e(0))=n$ and $p(e(1))=n+1$) and $p(v\Gamma_0)=\{p(v)\}$.  A conjugation 2-cell is mapped to an interval $[n,n+1]$. We say a vertex, edge or 2-cell is in {\it level} $n$ if $p$ maps it to $n$.  If $X$ is a 2-complex, let $X^{(1)}$ be the 1-skeleton of $X$, endowed with the edge path metric. Suppose $A$ is a subcomplex of $X$ and $A^{(1)}$ is the 1-skeleton of $A$. Let $B(A^{(1)},n)$ be the ball of radius $n$ in $X^{(1)}$ around $A^{(1)}$. For $n\geq 1$, let $St^n(A)$ be $B(A^{(1)},n)$ union all 2-cells whose boundaries belong to $B(A^{(1)},n)$. 
\medskip

\noindent ($\ast$) {\it Let $N_0$ be such that $|\hat \phi(h)|+3\leq N_0$ for all $h\in \mathcal H_1$. Then any (closed) 2-cell arising from a conjugation relation belongs to $St^{N_0}(v)$ for any vertex of that cell.} 

All homotopies of paths will be relative to end points and we simply say the paths are homotopic. 
Note that ($\ast$) implies:

\begin{lemma} \label{oldslide} {\bf (oldslide)}
Suppose $\alpha$ is an $\mathcal H_1$ path from $v$ to $w$ in $\Gamma$. Let $e$ and $d$ be the edges labeled $t$ at $v$ and $w$ respectively. Then the edge path $(e^{-1}, \alpha, d)$ is homotopic  to an $\mathcal H_0$ path $\beta$ by a homotopy $H$ with image in $St^{N_0}(im(\alpha))$. Furthermore $p(H)\subset [p(v),p(v+1)]
$. 
\end{lemma}
We apply the next lemma in both this section and in the proof of Lemma \ref{SS2}. 
\begin{lemma} \label{SSL1} {\bf (SSL1)}
Suppose $A_0$ and $A_1$ are finitely presented groups, and $A_1$ has finite index in $A_0$. Let $\mathcal Q=\langle \mathcal A_0:R\rangle$ be a finite presentation of $A_0$ such that $\mathcal A_1\subset \mathcal A_0$ generates $A_1$ (so that the Cayley graph $\Gamma(A_1,\mathcal A_1)$ is a subset of the Cayley 2-complex $\Gamma(\mathcal Q)$). 
There is an integer $N_{\ref{SSL1}}(\mathcal Q)$ such that if $\alpha$ is an edge path in $\Gamma(\mathcal Q)$ with $\alpha(0)=v$ and end point in $ v\Gamma(A_1,\mathcal A)$ then $\alpha$ is homotopic to an edge path in $v\Gamma(A_1,\mathcal A_1)$ by a homotopy $H$ in $\Gamma(\mathcal Q)$ with image in $St^{N_{\ref{SSL1}}(\mathcal Q)}(im(\alpha))$. 
\end{lemma}

\begin{proof} 
Say $\alpha=(e_1,\ldots, e_n)$ with vertices $v=v_0,\ldots, v_n$. If $i\in \{1,\ldots, n-1\}$ then there is an edge path $\phi_i$ of length $\leq \bar N_{\ref{FIclose}}=N_{\ref{FIclose}} (A_0,\mathcal A_0,A_1,\mathcal A_1)$ from $v_i$ to $w_i\in v\Gamma(A_1,\mathcal A_1)$. Define $\phi_0$ and $\phi_n$ to be trivial paths and define 
$w_0=v_0$ and $w_n=v_n$. Choose an integer $M$ such that if two vertices of $\Gamma(A_1,\mathcal A_1)$ are within $2\bar N_{\ref{FIclose}}+1$ in $\Gamma(\mathcal Q)$ then there is an edge path in $\Gamma(A_1,\mathcal A_1)$ of length $\leq M$ connecting these vertices. Choose $N_{\ref{SSL1}}(\mathcal Q)$ so that any edge path loop in $\Gamma(\mathcal Q)$ of length $\leq 2\bar N_{\ref{FIclose}}+M+1$ is homotopically trivial in $St^{N_{\ref{SSL1}}(\mathcal Q)}(a)\subset \Gamma(\mathcal Q)$ for any vertex $a$ of the loop.  For $i \in \{1,\ldots, n\}$  let $\tau_i$ be an $\mathcal H_1$ edge path (in $v\Gamma(A_1,\mathcal A_1)$) of length $\leq M$ from $w_{i-1}$ to $w_{i}$. Each loop $(\phi_i, \tau_{i+1}, \phi_{i+1}^{-1}, e_{i+1}^{-1})$ is homotopically trivial in $St^{N_{\ref{SSL1}}(\mathcal Q)}(v_i)\subset \Gamma(\mathcal Q)$. This implies that $\alpha$ is homotopic to $(\tau_1,\ldots, \tau_{n})$, by a homotopy in $St^{N_{\ref{SSL1}(\mathcal Q)}}(im(\alpha))\subset \Gamma(\mathcal Q)$.
\end{proof}

Next we recall some basic group theory for HNN-extensions. Suppose $w\in F(\mathcal H_0)$ and $q(w)=1\in G$. Then the number of letters of $w$ labeled $t$ is the same as the number of letters of $w$ labeled $t^{-1}$. If some letter of $w$ is labeled $t$, then (cyclically), there is a subword $t^{-1} w' t$ of $w$ such that $q(w')\in H_1$.  Suppose $\alpha$ is an edge path in $\Gamma$ with consecutive vertices $v_0,\ldots, v_n$. If the edge $(v_{i-1}, v_i)$ is labeled $a_i$ and is directed from $v_i$ to $v_{i+1}$ then let $b_i=a_i$. If  this  edge is directed $(v_{i},v_{i-1})$ then let $b_i=a^{-1}$. We say the word $b_1\cdots b_n$ ($\in F(\mathcal H_0\cup \{t\}$) {\it corresponds} to $\alpha$ and we write $\bar \alpha=b_1\cdots b_n$. 

An edge path has {\it backtracking} if there are adjacent edges $e$ and $e^{-1}$. 
The next lemma effectively allows us to move lowest level maximal subpaths of certain loops upward. Lemmas \ref{smallmove} and \ref{slideupward} are used in the simply connected at $\infty$ part of our result. The no backtracking hypothesis is not an issue for simply connected at $\infty$ arguments, but when a base ray must be respected in the semistability part of our theorem, we cannot assume there is (cyclically) no backtracking in certain loops. 

\begin{lemma} \label{smallmove} {\bf (smallmove}) 
Suppose $\gamma$ is an edge path loop in $\Gamma$ such that:  at least one edge of $\gamma$ is labeled $t$; $p( \gamma)=[A,B]$; and  (cyclically) $\gamma$ has no backtracking. Assume $\alpha$ is (cyclically) a maximal subpath of $\gamma$ such that $p(\alpha)=A$. Let $e$ and $d$ be the edges of $\gamma$ that (cyclically) immediately precede and follow $\alpha$, so that $\beta=(e,\alpha, d)$ is (cyclically) a subpath of $\gamma$. Then $q\bar \alpha\in H_1$, $\bar \beta=t^{-1}\bar \alpha t$ and $\beta$ is homotopic to an $\mathcal H_0$ path $\tau$ by a homotopy $H$ with image in $St^{N_0+N_{\ref{SSL1}}}(im(\alpha))$ and  with $p(H)\subset [ A, A+1]$. 
\end{lemma}

\begin{proof}
The maximality of $\alpha$ implies $\bar\beta=t^{-1}\bar\alpha t$.
We next argue that $q(\bar \alpha)\in H_1$. By the remarks preceding this lemma, there is (cyclically) a subword of $\bar\gamma$ of the form $t^{-1}\bar \beta t$ where  $q(\bar \beta)\in H_1$. If $\bar \beta=\bar \alpha$ then indeed $q(\bar \alpha)\in H_1$. Otherwise, $\bar \alpha$ and $\bar \beta$ do not overlap. Furthermore (since $\gamma$ has no backtracking) $e$ and $d$ are not adjacent to $\beta$ (see Figure \ref{Fig2}).
If $\bar \beta\ne \bar \alpha$ then replace $t^{-1}\bar \beta t$ in $\bar \gamma$ by an $\mathcal H_0$-word. This produces a word with two fewer $t$ letters, but the $t$ letters corresponding to $d$ and $e$ remain.  Continuing, we eventually conclude that $q(\bar \alpha)\in H_1$. Let $v$ be the initial vertex of $\alpha$, then the end point of $\alpha$ is in $v\Gamma_1$. 

\begin{figure}
\vbox to 3in{\vspace {-2in} \hspace {-.6in}
\hspace{-1 in}
\includegraphics[scale=1]{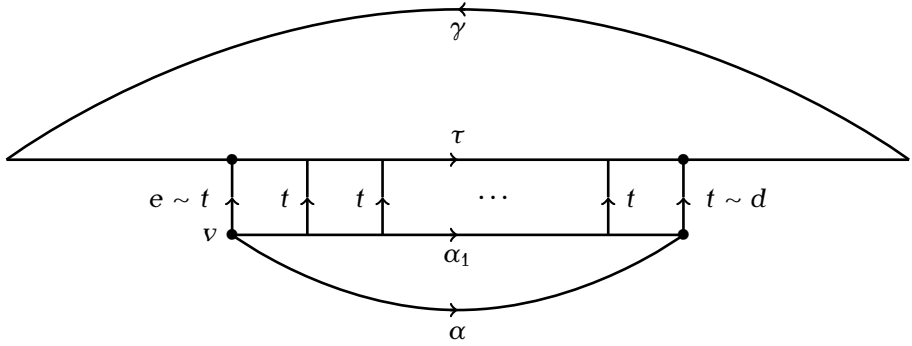}
\vss }
\vspace{-1.5in}
\caption{Sliding $\alpha$  upward to $\tau$}
\label{Fig2}
\end{figure}

By Lemma \ref{SSL1}, $\alpha$ is homotopic to an edge path $\alpha_1$ in $v\Gamma_1$ by a homotopy $\tilde H$ (in $v\Gamma_0$) with image in $St^{N_{\ref{SSL1}}}(im(\alpha))$. In particular, the image of $p\tilde H$ is $p(v)=A$.
Now we can slide $\alpha_1$ up one level (between $t$ labeled edges) to a path $\tau$, by a homotopy $\hat H$ which only uses 2-cells arising from conjugation relations on the letters of $\alpha_1$ (see Figure \ref{Fig2}). Note that the image of this homotopy lies between levels $A$ and $A+1$ and $\hat H$ has image in $St^{N_0}(im(\alpha_1))$. Combine $\tilde H$ and $\hat H$ to form the homotopy $H$ and finish the proof. 
\end{proof}

\begin{lemma}\label{slideupward} {\bf (slideupward)} 
Suppose $C$ is a compact subcomplex of $\Gamma$, $p(C)\subset [A,B]$ and $\gamma$ is an edge path loop (without backtracking)  in $\Gamma-C$. Then either $\gamma$ is homotopically trivial in $\Gamma-C$ or there is a homotopy 
in $\Gamma-C$ of $\gamma$ to an edge path loop $\gamma_1$ such that $p(\gamma_1)$ has image in $[A-1,\infty)$, the homotopy fixes the edges of $\gamma$ in levels $A-1$ and above and the only edges of $\gamma_1$ above level $A-1$ (including $t$-labeled edges connecting levels $A-1$ and $A$) are edges of $\gamma$. 
\end{lemma}
\begin{proof} 
Let $\bar\gamma \ (\in F(\mathcal H_0\cup \{t\}))$ be the (reduced) word labeling $\gamma$. If $\gamma$ has image in a single level, then either $\gamma$ lies in a level at or above level $A-1$ (and we are finished) or $\gamma$ lies in a level below level $A-1$. In the later case, $\gamma$ is homotopically trivial in that level (by a homotopy avoiding $C$).  Instead, we assume some letter of $\bar \gamma$ is $t$. 

Suppose $\alpha$ is (cyclically) a maximal subpath of $\gamma$ in the lowest level of $\Gamma$ traversed by $\gamma$. If $p(\alpha)\geq A-1$ we are finished. If $e$ and $d$ are the edges preceding and following $\alpha$, then Lemma \ref{smallmove} implies that $(e,\alpha,d)$ is homotopic to an $\mathcal H_0$ path $\tau$ (with $p(\tau)=p(\alpha)+1$) by a homotopy $H$ with $p(H)\subset (-\infty,A-1]$. In particular, $H$ avoids $C$. Replace $(e,\alpha,d)$ in $\gamma$ by $\tau$, eliminate backtracking and continue applying Lemma \ref{smallmove} to (cyclically) maximal subpaths in a lowest level until all edges of the resulting path are either in a single level below level $A-1$ (in which case $\gamma$ is homotopically trivial in $\Gamma-C$) or in levels $A-1$ and above.
\end{proof}

Now assume that $H_0$ is 1-ended. In this case, our goal is to show that $G$ is simply connected at $\infty$. Let $C$ be a finite subcomplex of $\Gamma$ and assume that $p(C)\subset [A, B]$.  Let $D$ be a finite subcomplex of $\Gamma$ containing  $St^{(B-A+2)(N_{\ref{SSL1}} +N_0)}(C)$. Also assume that if $g\in G$  then $g\Gamma_0-D$ is a single unbounded component (since $H_0$ is one ended). 
Let $\gamma$ be an edge path loop in $\Gamma-D$. It suffices to show that $\gamma$ is homotopically trivial in $\Gamma-C$.   The idea is to move $\gamma$ to a loop in a single level above level $B$ by a homotopy that avoids $C$. Once there, it is homotopically trivial in that level and hence that homotopy avoids $C$. Suppose $p(\gamma)=[A_1,B_1]$.

Inductively we show that for $0\leq k\leq B-A+2$, the loop $\gamma$ is either homotopically trivial in $\Gamma-C$ (as desired) or homotopic to an edge path loop $\mu_k$ by a homotopy in $\Gamma-C$ such that $\mu_k$ has image in $\Gamma-St^{(B-A+2-k)(N_{\ref{SSL1}} +N_0)}(C)$ and $p(\mu_k)\subset [A-1+k,\infty)$. Note that for $k=B-A+2$, we have $p(\mu_k)\subset [B+1,\infty)$ (so $\mu_k$ is above $C$). 
\begin{figure}
\vbox to 3in{\vspace {-1.5in} \hspace {.1in}
\hspace{-1 in}
\includegraphics[scale=1]{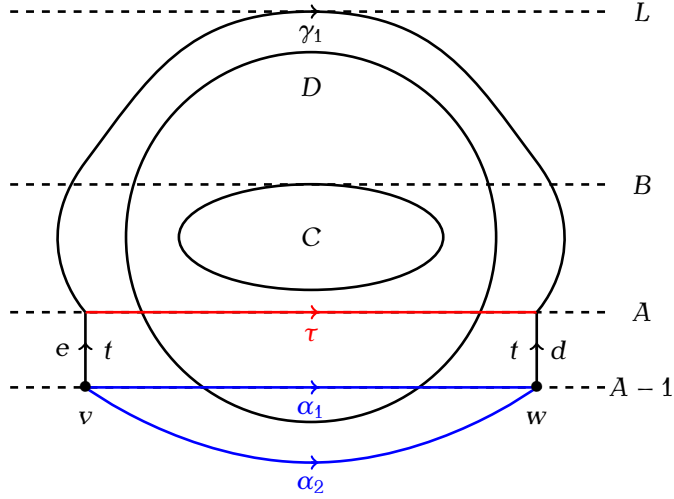}
\vss }
\vspace{-.2in}
\caption{Sliding $\gamma$ Up}
\label{Fig3}
\end{figure}
If $A_1\geq A-1$, then let $\mu_0=\gamma$. If $A_1<A-1$ then Lemma \ref{slideupward} implies that either $\gamma$ is homotopically trivial in $\Gamma-C$ (and we are finished) or there is a homotopy
in $\Gamma-C$ of $\gamma$ to an edge path loop $\gamma_1'$ such that $p(\gamma_1')\subset [A-1,\infty)$, the homotopy fixes the edges of $\gamma$ in levels $A-1$ and above, and the only edges of $\gamma_1'$ above level $A-1$ (including $t$-labeled edges connecting levels $A-1$ and $A$) are edges of $\gamma$. If $\gamma_1'$ has image in level $A-1$ then it is homotopically trivial in that level (hence missing $C$) and again we are finished. Let $\gamma_1$ be obtained from $\gamma_1'$ by eliminating backtracking (which can only occur in level $A-1$). If $\alpha_1$ is a maximal subpath of $\gamma_1$ in level $A-1$ and $e$ and $d$ are the edges of $\gamma_1$ that (cyclically) immediately precede and follow $\alpha_1$, then $e$ and $d$ are edges of $\gamma$ labeled $t$. In particular if $v$ and $w$ are the initial and end points of $\alpha_1$ then $v$ and $w$ are in the single unbounded component of $v\Gamma_0-D$. Hence there is an $\mathcal H_0$ path $\alpha_2$ (without backtracking) in $v\Gamma_0-D$ from $v$ to $w$. Now $\alpha_1$ and $\alpha_2$ are homotopic by a homotopy in $v\Gamma_0$ (in level $A-1$) and hence in $\Gamma-C$ (see Figure \ref{Fig3}). Let $\mu_0$ be obtained from $\gamma_1$ by replacing each maximal $\mathcal H_0$ subpath $\alpha_1$ of $\gamma_1$ with a corresponding $\alpha_2$. Note that there is no backtracking in $\mu_0$.  We have that $\gamma$ is homotopic to $\mu_0$ by a homotopy in $\Gamma-C$, the image of $\mu_0$ is in $\Gamma-D$ and $p(\mu_0)\subset [A-1,\infty)$. 

If $A_1>A-1$, let $\mu_1=\gamma$. Otherwise let $\alpha_2$ be a maximal subpath of $\mu_0$ in level $A-1$. Let $e$ and $d$ be the edges of $\mu_0$ that (cyclically) immediately precede and follow $\alpha_2$ (see Figure \ref{Fig3}). 
The path $\beta=(e,\alpha_2, d)$ is such that $\bar\beta=t^{-1}\bar\alpha_2 t$. 
Lemma \ref{smallmove} implies that
$\beta$  is homotopic to an $\mathcal H_0$ path $\tau$ by a homotopy $H$ with image in $St^{N_0+N_{\ref{SSL1}}}(im(\alpha_2))$ and  with $p(H)\subset [ A-1, A]$. In particular $H$ (and $\tau$) avoids $St^{(B-A+1)(N_{\ref{SSL1}} +N_0)}(C)$. For all such maximal subpaths $\alpha_2$ of $\mu_0$ in level $A-1$ replace $(e,\alpha_2,d)$ by the corresponding $\tau$. Eliminate backtracking (all in level $A$) and call the resulting path $\mu_1$.
The result is that $\mu_0$ (and $\gamma$) is homotopic to an edge path loop $\mu_1$ with image in $\Gamma-St^{(B-A+1)(N_{\ref{SSL1}} +N_0)}(C)$ by a homotopy in $\Gamma-C$ and $p(\mu_1)\subset [A,\infty)$.   
If $A_1>A$, let $\mu_2=\gamma$. Otherwise,
we can obtain $\mu_2$ by apply Lemma \ref{smallmove} to each maximal subpath of $\mu_1$ in level $A$, as long as $\mu_1$ does not have image in level $A$. If $\mu_1$ has image in level $A$, then $\mu_1$ is an $\mathcal H_0$ loop at say the vertex $v$. Lemma \ref{SSL1} implies that $\mu_1$ is homotopic to an $\mathcal H_1$ edge path loop $\mu_1'$ at $v$ by a homotopy in $St^{N_{\ref{SSL1} }}(im(\mu_1))$. If $e$ is the edge labeled $t$ at $v$ then Lemma \ref{oldslide} implies that $(e^{-1},\mu_1', e)$ is homotopic to an $\mathcal H_0$ path by a homotopy $H$ with image in $St^{N_0}(im(\mu_1'))$ and such that $p(H)\subset [p(v), p(v)+1]$. This means that $(e^{-1}, \mu_1, e)$ is homotopic to an $\mathcal H_0$ path in level $A+1$ by a homotopy in $\Gamma-St^{(B-A)(N_{\ref{SSL1}}+{N_0}}(C)$. If $\mu_1$ is not in level $A$ we apply Lemma \ref{smallmove} to maximal subpaths of $\mu_1$ in level $A$, so that in either case we have that $\mu_1$ is homotopic to an edge path loop $\mu_2$ (without backtracking) with image avoiding  $St^{(B-A)(N_{\ref{SSL1}}+{N_0}}(C)$, by a homotopy in $\Gamma-C$ and $p(\mu_2)\subset [A+1,\infty)$.

Repeating this procedure $k=B-A+2$ times (starting with $\mu_0$) we obtain an edge path loop $\mu_k$ (without backtracking) with $p(\mu_k)$ either equal to $B+1$ (if $B_1\leq B+1$) or a subset of $[B+1,B_1]$ (if $B_1>B+1$). We have $\gamma$ is homotopic to $\mu_k$ by a homotopy in $\Gamma-C$. In the latter case, use Lemma \ref{smallmove} to move lowest level maximal $\mathcal H_0$ subpaths of $\mu_k$ (and of the resulting loops) up one level by homotopies with image above level $B$ and (hence avoiding $C$) until a loop in  a single level is obtained. This loop is homotopically trivial in that level. {\bf This completes the proof that if $H_0$ is 1-ended then $G$ is simply connected at $\infty$.} 

\medskip

Now we turn to the semistability part of our theorem. 
The next lemma allows us to replace a path with labeling $(t^{-1}, \bar \alpha,t)$ (where $\bar\alpha$ is an $\mathcal H_0$ word) with a path comprised of  subpaths either in the level above the level of $\alpha$ or in the compliment of an arbitrary compact set $E$. In Figure \ref{Fig4} the path $(e^{-1}, \alpha, d)$ is homotopic to $(\gamma_1', c^{-1},\tau_2, b, \gamma_2'^{-1})$ where $p(\gamma_i')=p(\alpha)+1$ and the path $(c^{-1},\tau_2, b)$ has image in $\Gamma-E$.

\begin{lemma} \label{UpOut} {\bf (UpOut)}
Suppose: 
\begin{itemize} 
\item [$A.$] $D$ is a compact subcomplex of $\Gamma$ such that for each vertex $v\in \Gamma$ the space $v\Gamma_0-D$ is a union of unbounded path components, and
\item [$B.$] $\beta=(e^{-1},\alpha, d)$ is an edge path of $\Gamma$ such that $\bar\beta=t^{-1}\bar\alpha t$,  $e$ and $d$ are edges of $\Gamma-D$, $\alpha$ is an $\mathcal H_0$ path and the end points of $\alpha$ can be joined by a path $\alpha'$ in $\Gamma-D$ such that $p(\alpha'(t))\leq p(\alpha)$ for all $t$. 
\end{itemize}
\noindent Then, for any compact set $E\subset \Gamma$ such that $D\subset E$, there is a homotopy $M$ of $\beta$  to a path $(\gamma_1', c^{-1} , \tau_2, b, \gamma_2'^{-1})$ such that: 
\begin{enumerate}
\item $\tau_2$, $\gamma_1'$ and $\gamma_2'$ are $\mathcal H_0$ paths, 
\item $b$ and $c$ are $t$ edges, 
\item $(c^{-1}, \tau_2, b)$ has image in $\Gamma-E$, 
\item for $i\in\{1,2\}$, the end points of $\gamma_i'$ can be connected by an edge path $\gamma_i''$ with image in $\Gamma-D$ such that $p(\gamma_i''(t))\leq p(\gamma_i')$ for all $t$, and
\item$p(M)$ has image in $[p(\alpha), p(\alpha)+1]$. 
\end{enumerate}
\end{lemma}

\begin{proof} Suppose $\rho_1$ and $\rho_2$ are paths with the same end points. 
We use the notation $\rho_1 \buildrel K\over \sim \rho_2$ when $K$ is a homotopy of $\rho_1$ to $\rho_2$ relative to end points.
Say the initial (terminal) vertex of $\alpha$ is $v$ ($w$). Since $v\Gamma_1\subset v\Gamma_0$ and $w\in v\Gamma_0-D$, hypothesis ({\it A.}) implies that there is an $\mathcal H_0$ path $\tau_1$ at $w$ with image in $v\Gamma_0-D$ such that  the end point $z$ of $\tau_1$ belongs to $v\Gamma_1-St^{N_{\ref{FIclose}}+1} (E)$. (See Figure \ref{Fig4}.) Let $\gamma_1$ be an $\mathcal H_1$ path from $v$ to $z$ and $K_1$ a homotopy in $v\Gamma_0$ realizing: 
$$(\alpha, \tau_1)\buildrel K_1\over \sim\gamma_1.$$ 
Lemma \ref{FIclose} implies there is an $\mathcal H_0$ edge path $\tau_2$ of length $\leq N_{\ref{FIclose}}$ from $z$ to $y\in w\Gamma_1$.  Let $\gamma_2$ be an $\mathcal H_1$ path from $w$ to $y$ and $K_2$ a homotopy in $w\Gamma_0$ realizing:
 $$(\tau_1,\tau_2)\buildrel K_2\over \sim\gamma_2.$$ 
Let $c$ and $b$ be the $t$ edges at $z$ and $y$ respectively. Since $z\in \Gamma-St^{N_{\ref{FIclose}}+1} (E)$ we have that $(c^{-1},\tau_2,b)$ has image in $\Gamma-E$.  At this point conclusions ({\it 2.}) and ({\it 3.}) are satisfied. By Lemma \ref{oldslide}, there are $\mathcal H_0$ paths $\gamma_1'$ and $\gamma_2'$ such that: 
$$(e^{-1},\gamma_1, c)\buildrel L_1\over \sim\gamma_1'\hbox{ and }(d^{-1},\gamma_2,b)\buildrel L_2\over \sim\gamma_2'$$
with $p(L_i)\subset [p(\alpha),p(\alpha)+1]$. Conclusion ({\it 1.}) is satisfied. Hypothesis ({\it B.}) and the definitions of $\tau_1$ and $\tau_2$ imply that  $\gamma_1''=(e^{-1}, \alpha',\tau_1, c)$ is a path in $\Gamma-D$ connecting the end points of $\gamma_1'$, and $\gamma_2''=(d^{-1}, \tau_1,\tau_2,b)$ is a path in $\Gamma-D$ connecting the end points of $\gamma_2'$. Conclusion ({\it 4.}) is satisfied. Combining $K_i$ and $L_i$  to form the homotopy $M_i$ we have: 
$$(e^{-1},\alpha,\tau_1,c)\buildrel M_1\over \sim\gamma_1'\hbox{ and }(d^{-1},\tau_1,\tau_2,b)\buildrel M_2\over \sim\gamma_2'$$ 
with $p(M_i)= [p(\alpha), p(\alpha)+1]$. 
\begin{figure}
\vbox to 3in{\vspace {-3in} \hspace {-1in}
\hspace{-1 in}-1\includegraphics[scale=1]{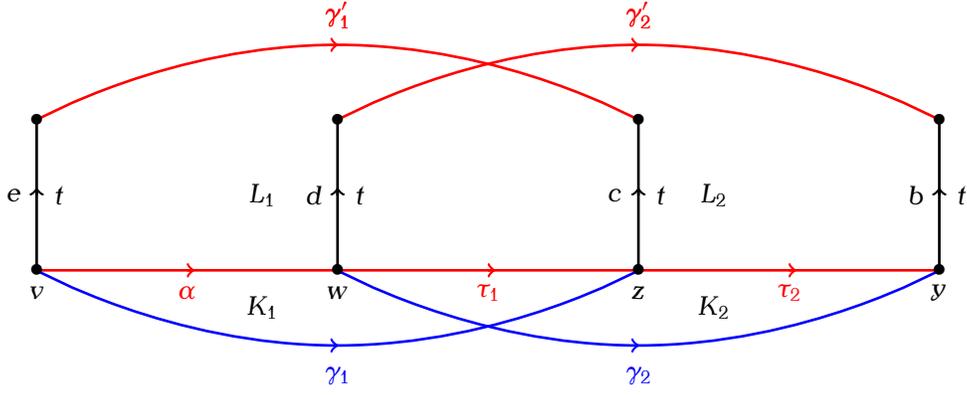}
\vss }
\vspace{-2in}
\caption{Sliding $\beta =(e^{-1}, \alpha, d)$ up and out}
\label{Fig4}
\end{figure}
The homotopy $M_2$ implies  that: 
$$\tau_1\buildrel M_2'\over \sim (d,\gamma_2',b^{-1},\tau_2^{-1}).$$ 
Substituting for $\tau_1$ in $M_1$ we have: 
$$(e^{-1},\alpha,d,\gamma_2',b^{-1},\tau_2^{-1},c)\sim \gamma_1'$$
and so: 
$$(e^{-1},\alpha,d)\buildrel M\over \sim (\gamma_1',c^{-1},\tau_2,b,\gamma_2'^{-1})$$
where $p(M)= [p(\alpha),p(\alpha)+1]$, 
so that ({\it 5.}) is satisfied.
\end{proof}

Use of the next lemma is one of the primary difference between the proof of (Theorem 3.1, \cite{HNN1}) and the current argument. It is one of the main ingredients in the proof of Theorem \ref{MMFI}. If one thinks of $t$-labeled edges of $\Gamma$ as vertical and $\mathcal H_0$ labeled edges as horizontal, then Lemma \ref{push} explains how to push a vertical edge $e$ that is ``far enough" from a compact set $C\subset \Gamma$, to a vertical edge $e_1$ that is ``arbitrarily far" from $C$, by a homotopy that misses $C$.

\begin{lemma} \label{push} {\bf (push)} 
Given a compact subcomplex $C$ of $\Gamma$ there is a compact subcomplex $D_{\ref{push}}(C)$ of $\Gamma$ containing $St^{N_0+N_{\ref{SSL1}}}(C)$ such that for any third compact set $E$ and any edge $e$ labeled $t$ in $\Gamma-D_{\ref{push}}(C)$, there is an edge path $(\alpha, e_1,\beta)$ from $e(0)$ to $e(1)$ such that $\alpha$ and $\beta$ are $\mathcal H_0$  paths, the edge $e_1$ has image in $\Gamma-E$ and label $t$, and $e$ is homotopic to $(\alpha, e_1,\beta)$ by a homotopy with image in $\Gamma-C$. 
\end{lemma}
\begin{proof} Note that if $w$ and $v$ are vertices of $\Gamma$ and $w\in v\Gamma_0$ then $w\Gamma_0=v\Gamma_0$. For $K$ a subcomplex of $\Gamma$, we take $St(K)$ in $\Gamma$. (See $(\ast)$ for the definition of $N_0$.) There are only finitely many $v\Gamma_0$ that intersect $St^{N_{\ref{SSL1}}+N_0}(C)$. Choose $D$ a finite subcomplex of $\Gamma$ such that if $v\Gamma_0\cap St^{N_0+N_{\ref{SSL1}}}(C)\ne\emptyset$, then $St^{N_0+N_{\ref{SSL1}}}(C)$ union all bounded components of $v\Gamma_0-St^{N_0+N_{\ref{SSL1}}}(C)$ belong to $D$ (see Theorem \ref{ES}).  Let $v=e(0)$. If $v\in \Gamma-D$ then by the definition of $D$, there are edge paths in $v\Gamma_0-St^{N_0+N_{\ref{SSL1}}}(C)$ from $v$ to vertices arbitrarily far from $C$. In particular there is an edge path $\alpha_1$ in $v\Gamma_0-St^{N_0+N_{\ref{SSL1}}}(C)$ from $v$ to $w_1\in v\Gamma_0-St^{N_{\ref{FIclose}}+1}(E\cup D)$. By Lemma \ref{FIclose} there is an edge path $\alpha_2$ (of length $\leq N_{\ref{FIclose}}$) from $w_1$ to $w\in  v\Gamma_1-St(E\cup D)$. In particular, $\alpha_2$ avoids $D$. Then $\alpha=(\alpha_1,\alpha_2)$ has image in $v\Gamma_0-St^{N_0+N_{\ref{SSL1}}}(C)$. By Lemma \ref{SSL1}, $\alpha$ is homotopic to an edge path $\alpha'$ (with image in $v\Gamma_1$) by a homotopy in $v\Gamma_0-St^{N_0}(C)$. 
Let $e_1$ be the edge labeled $t$ at $w$. Since $w\in v\Gamma_0-St(E)$,  $e_1$ avoids $E$. Combining conjugation 2-cells (one for each edge of $\alpha'$), the path $\alpha'$ is homotopic to an edge path $(e,\hat \beta,e_1^{-1})$ by a homotopy in $St^{N_0}(im(\alpha'))\subset \Gamma-C$. Let $\beta=\hat \beta^{-1}$. 
\end{proof}

Lemma \ref{push} allows us to focus on edge paths in $v\Gamma_0$ ($v\in G$) with end points outside of a compact set $E$. 
Theorem \ref{ES} implies that for any compact set $E\subset \Gamma$ and $v\in G$, $E$ union all bounded components of $v\Gamma_0-E$ is compact. Since $E$ only intersects finitely many $v\Gamma_0$ non-trivially,  $E$ is contained in a compact set $E'$ such that for any $v\in G$ the set 
$v\Gamma_0-E$ is a union of unbounded components. 

The point of the next lemma is to show the path $\gamma$ is homotopic  to a path comprised of several $H_0$ subpaths in $\Gamma-E$ and a path $\gamma_1$ ``one level higher" than the path $\gamma$. It will be inductively applied in Lemma \ref{slideC}.

\begin{remark}\label{Rslideup} {\bf (Rslideup)} 
Lemma \ref{slideup} is valid when $F=\emptyset$. In this case, $St^k(F)=\emptyset$ for all $k\geq 0$. We will applied this special case to paths $\gamma$ in a level above the highest level of a certain compact set $C\subset \Gamma$. In this case the homotopy $H$ of Lemma \ref{slideup} avoids $C$ since $p(H)$ has image above the highest level of $C$.
\end{remark}

\begin{lemma} \label{slideup} {\bf (slideup)} 
Suppose $E$ and $F$ are finite subcomplexes of $\Gamma$, such that $E$ contains $St^{N_{\ref{SSL1}} +N_0}(F)$, and  $g\Gamma_0-E$ is a union of unbounded components for any $g\in G$.  If $\gamma$ is an edge path in $v\Gamma_0-St^{N_{\ref{SSL1}} +N_0}(F)$ with end points in $\Gamma-E$, then there is a homotopy $H$ with image in $\Gamma-F$, of $\gamma$  to an edge path $(\tau, e, \gamma_1, e'^{-1}, \tau'^{-1})$ where $e$ and $e'$ are edges labeled $t$; $\tau$ and $\tau'$ are $\mathcal H_0$ paths in $\Gamma-E$; and $ \gamma_1$ is an $\mathcal H_0$ path in $\Gamma-F$ (so that $p(\gamma)+1=p(\gamma_1)$) with end points in $\Gamma-E$. Furthermore $p(H)$ has image in $[p(v), p(v)+1]$.
\end{lemma}

\begin{proof} 
In Figure \ref{Fig5}, the path $(\tau_1,\tau_2)$ will be our path $\tau$ and combining the two homotopies $K_1$ and $K_2$ will give the homotopy $H$. Let $v$ be the initial and $w$ the terminal vertex of $\gamma$. Since $w\Gamma_0-E$ is a union of unbounded components, there is an $\mathcal H_0$ path $\tau'$ in $w\Gamma_0-E$ from $w$ to a vertex $w_1\in w\Gamma_0-St(E)$. See Figure \ref{Fig5}. 
There is an $\mathcal H_0$ path $\tau_1$  in $v\Gamma_0-E$ from $v$ to $v_1\in v\Gamma_0-St^{N_{\ref{FIclose}}+1}(E)$. Note that:
$$v\Gamma_0=w\Gamma_0=v_1\Gamma_0=w_1\Gamma_0.$$

\begin{figure}
\vbox to 3in{\vspace {-2in} \hspace {-1.2in}
\hspace{-1 in}-1\includegraphics[scale=1]{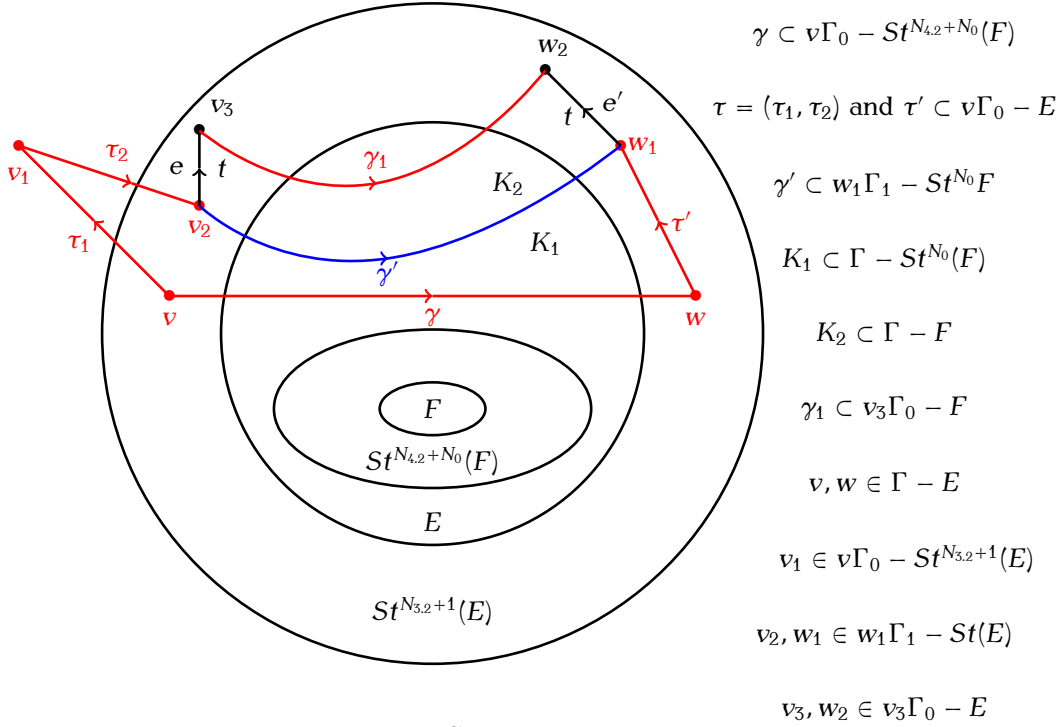}
\vss }
\vspace{.2in}
\caption{Sliding $\gamma$  upward to $\gamma_1$}
\label{Fig5}
\end{figure}

By Lemma \ref{FIclose} there is an $\mathcal H_0$ path  $\tau_2$  of length $\leq N_{\ref{FIclose}}$ from $v_1$ to a vertex $v_2$ of $w_1\Gamma_1$. Note that $im(\tau_2)\subset v_1\Gamma-St(E)$ (and so $v_2\in w_1\Gamma_1-St(E)$).
By Lemma \ref{SSL1} the path $\nu=(\tau_2^{-1},\tau_1^{-1},\gamma,\tau'_1)$ is homotopic to an $\mathcal H_1$ path $\gamma'$, by a homotopy $K_1$ with image in $St^{N_{\ref{SSL1}}}(im(\nu))$. Since $E$ contains $St^{N_{\ref{SSL1}} +N_0}(F)$ and $\gamma$ avoids $St^{N_{\ref{SSL1}} +N_0}(F)$, the path $\nu$ avoids $St^{N_{\ref{SSL1}} +N_0}(F)$. Hence $K_1$  avoids $St^{N_0}(F)$. In particular, $\gamma'$ avoids $St^{N_0}(F)$.

Let $e$ be the edge labeled $t$ at $v_2$ (with end point $v_3$) and let $e'$ be the edge labeled $t$ at $w_1$ (with end point $w_2$). By Lemma \ref{oldslide}, the path $\gamma'$ is homotopic to $(e,\gamma_1, e'^{-1})$ (where $\gamma_1$ is an $\mathcal H_0$ path) by a homotopy $K_2$ with image in $St^{N_0}(im(\gamma'))$. Since $\gamma'$ avoids $St^{N_0} (F)$, $K_2$ (and $\gamma_1$) avoids $F$. Since $v_2,w_1\in \Gamma-St(E)$, we have $v_3,w_2\in \Gamma- E$. Now combine $K_1$ and $K_2$ to produce $H$ and let $\tau=(\tau_1,\tau_2)$ to finish the proof.
\end{proof}

\begin{lemma}\label{slideC} {\bf (slideC)}
Suppose $C$ and $E$ are finite subcomplexes of $\Gamma$ such that $p(C)\subset [A,B]$ and $E$  contains $St^{(B-A+2)(N_{\ref{SSL1}} +N_0)}(C)$. Also assume that if $g\in G$  then $g\Gamma_0-E$ is a union of unbounded components. If $\gamma$ is an edge path in $v\Gamma_0-St^{(B-A+2)(N_{\ref{SSL1}} +N_0)}(C)$, with end points in $\Gamma-E$ and $p(v)\in [A-1,B]$ then there is a homotopy of $\gamma$ (with image in $\Gamma-C$) to a path in $\Gamma-E$. 
\end{lemma}
\begin{proof}
Apply Lemma \ref{slideup} to $\gamma$ with $F=St^{(B-A+1)(N_{\ref{SSL1}} +N_0)}(C)$. This gives a homotopy $H_1'$  with image in $\Gamma-F$ of $\gamma$ to  $(\beta_1,\gamma_1,\beta_1')$ (here $\beta_1=(\tau,e)$ and $\beta_1'=(e'^{-1}, \tau'^{-1})$) where $\beta_1$ and $\beta_1'$ have image in $\Gamma-E$ and $p(\gamma)+1=p(\gamma_1)$. Next, apply Lemma \ref{slideup} to $\gamma_1$ with $F=St^{(B-A)(N_{\ref{SSL1}} +N_0)}(C)$ to produce $(\beta_2,\gamma_2, \beta_2')$. So after applying Lemma \ref{slideup} at most $B-A+2$ times, we have a homotopy $H'$ with image in $\Gamma-C$ of $\gamma$  to $(\beta,\tilde \gamma, \beta')$ (the paths $\beta$ and $\beta'$ are the concatenations of the $\beta_i$ and $\beta_i'$ respectively) where $\beta$ and $\beta'$ have image in $\Gamma-E$ and $p(\tilde \gamma)=B+1$. Of course the image of $\tilde \gamma$ may intersect $E$, but $\tilde \gamma$ has image in level $B+1$ (above $C$) and we may apply Lemma \ref{slideup} to $\tilde\gamma$ with $F=\emptyset$ (see Remark \ref{Rslideup}), to obtain a homotopy $\tilde H$ of $\tilde \gamma$ to $(\alpha_1, \tilde \gamma_1, \alpha_1')$ where $\alpha_1$ and $\alpha_1'$ are $\mathcal H_0$ paths in $\Gamma-E$, $p(\tilde \gamma_1)=B+2$ and $\tilde H$ has image in $[B+1,B+2]$ (and hence $\tilde H$ avoids $C$). Continue applying Lemma \ref{slideup} (with $F=\emptyset$) to obtain $\alpha_2,\ldots, \alpha_n$ and $\alpha_2',\ldots, \alpha_n'$ all with image in $\Gamma-E$ and path $\tilde \gamma_n$ with $p(\tilde \gamma_n)=B+n+1$ larger than the largest integer in $p(E)$. Let $\alpha=(\alpha_1,\ldots, \alpha_n)$, $\alpha'=(\alpha_1',\ldots, \alpha_n')$. We have $\gamma$ is homotopic by a homotopy with image in $\Gamma-C$ to $(\beta, \alpha, \tilde\gamma_n, \alpha',\beta')$, a path with image in $\Gamma-E$. 
\end{proof}

Let $C_1,C_2,\ldots$ be a collection of compact subcomplexes of $\Gamma$ such that $C_i$ is a subset of the interior of $C_{i+1}$ and $\cup_{i=1}^\infty C_i=\Gamma$. If  $a$ and $b$ are vertices of $v\Gamma_0$ and $p(v) =n$, then define the {\it degree} of $(a,b)$ (written $deg(a,b)$) to be the largest integer $m$ such that there is an edge path $\alpha$ in $\Gamma-C_m$ from $a$ to $b$ such that $p(\alpha)\subset (-\infty, n]$.

The next result is in analogy with Lemma 4.4 of \cite{HNN1}. The proof is elementary, but it is one of the main ideas of the overall argument. The other lemmas used in the proof of Theorem \ref{MMFI}  explain how to carefully slide $H_0$ paths upward and $t$-edges horizontally. Lemma \ref{Like4.4} will allow us to slide certain $H_0$ paths downward. It is in the proof of Lemma \ref{Like4.4} where we need $H_0$ to be finitely presented (as opposed to only finitely generated). 

\begin{lemma}\label{Like4.4} {\bf (Like4.4)} 
Suppose $D$ is a compact subcomplex of $\Gamma$ and $W$ and $L$ are unbounded path components of $v\Gamma_0-D$, (we allow $W=L$ as a possibility) then one of the following two statements holds:

(i) The set $W\times L$ contains a collection of pairs of vertices $(w_1,l_1)$, $(w_2,l_2),\ldots$ such that all $w_j$ are distinct, all $l_j$ are distinct and $deg(w_j,l_j)\geq j$, or

(ii) There are finite sets of vertices $S\subset W$ and $T\subset L$ such that all pairs of vertices of $(W-S)\times (L-T)$ have degree less than some fixed positive integer $N_1(W, L)$. 
\end{lemma}

\begin{proof}
Begin selecting pairs of vertices $(w_j,l_j)$ satisfying the hypothesis of $(i)$. If for some $j\geq 1$, $(w_j,l_j)$ cannot be selected to satisfy hypothesis $(i)$, then all pairs of vertices of $(W-\{w_1,w_2,\ldots, w_{j-1})\times (L-\{l_1,l_2,\ldots, l_{j-1})$ have degree less than $j$ (so that $(ii)$ is satisfied).
\end{proof}

We are ready to combine our lemmas to prove the semistability part of Theorem \ref{MMFI}. Let $C$ be a finite subcomplex of $\Gamma$ and $r$ be the proper ray at $\ast$, each of whose edges is labeled $t$. Choose integers $A$ and $B$ such that
$$p(C)\subset [A,B].$$ 

\begin{figure}
\vbox to 3in{\vspace {-2.5in} \hspace {-.1in}
\hspace{-1 in}
\includegraphics[scale=1]{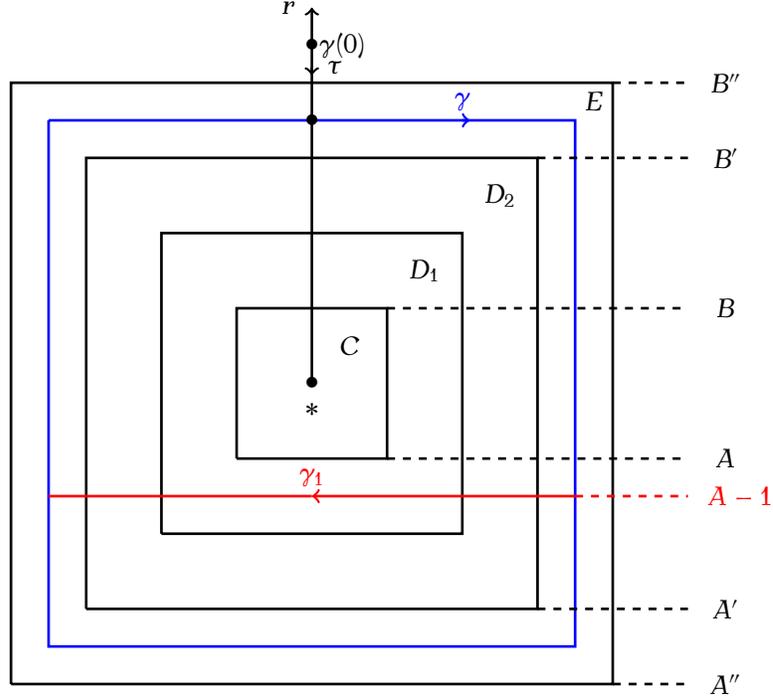}
\vss }
\vspace{-.3in}
\caption{The General Setup}
\label{Fig6}
\end{figure}

Let $D_1$ be a finite subcomplex of $\Gamma$ containing $St^{(B-A+2)(N_{\ref{SSL1}} +N_0)}(C)$. Furthermore,  if $v\Gamma_0$ intersects $D_1$ non-trivially and $p(v)=A-1$ (there are only finitely many of these), then we require that all components of $v\Gamma_0-D_1$ are unbounded. Next we define a finite subcomplex $D_2$ of $\Gamma$. Our first requirement is that $D_2$ contains $D_{\ref{push}}(D_1)$ (see Figure \ref{Fig6}). Suppose $v\Gamma_0$ intersects $D_1$ non-trivially, $p(v)=A-1$ and that $W$ and $L$ are unbounded components of $v\Gamma_0-D_1$ satisfying Lemma \ref{Like4.4} (ii) (where $D_1$ plays the role of $D$). Then we require that $D_2$ contains the finite sets $S\subset W$ and $T\subset L$ and the set $C_{N_1(W,L)}$ given by Lemma \ref{Like4.4} (recall that the $C_i$ are a cofinal sequence of finite subcomplexes of $\Gamma$). We also require that for each vertex $v\in \Gamma$ the space $v\Gamma_0-D_2$ is a union of unbounded components. An important fact follows immediately from Lemma \ref{Like4.4}:

\medskip

\noindent $(\ast)$ Suppose $v\Gamma_0$ intersects $D_1$ non-trivially, $p(v)=A-1$, and $\alpha$ is an edge path in $\Gamma-D_2$ that begins and ends in $v\Gamma_0$ such that $p(\alpha)\subset (-\infty, A-1]$. If the initial vertex of $\alpha$ belongs to the (unbounded) component $W$ of $v\Gamma_0-D_1$ and the terminal vertex of $\alpha$ belongs to the (unbounded) component $L$ of $v\Gamma_0-D_1$, then $W\times L$ satisfies condition $(i)$ of Lemma \ref{Like4.4}. 

\medskip

Assume 
$$p(D_2)\subset [A',B']\hbox{ and } r([0,B'])\subset D_2.$$ 
Let $E$ be an arbitrary finite subcomplex of $\Gamma$ (containing $D_2$). Enlarge $E$ so that for any vertex $v\in \Gamma$, $v\Gamma_0-E$ is a union of unbounded path components. If $\gamma$ is an edge path loop in $\Gamma-D_2$ and based on $r$, our goal is to show that $\gamma$ is homotopic $rel\{r\}$ to an edge path loop in $\Gamma-E$ by a homotopy in $\Gamma-C$ (see Theorem \ref{ssequiv}(2)). Say $\gamma (0)=r(K)$ for some $K>B'$. If $\tau$ is an edge path ending at $r(K)$, each of whose edges is labeled $t^{-1}$, then $\gamma$ is homotopic $rel\{r\}$ by a homotopy in $\Gamma-D_2$ to $(\tau, \gamma, \tau^{-1})$ (see Figure \ref{Fig6}). In particular, we may assume: 
$$p(E)\subset [A'',B'']) \hbox{ and } \gamma(0) =r(K)\hbox{ for }K>B''.$$  

We cannot assume that $\gamma$ is cyclically without backtracking. It may be that $\gamma$ has the form $(\tau, \gamma',\tau')$ where cyclically $\tau$ and $\tau'$ backtrack. If we eliminating $\tau$ and $\tau'$ our path $\gamma'$ may not be based on $r$. Unfortunately, we cannot apply Lemmas \ref{smallmove} or \ref{slideupward}.
Instead, we will apply Lemma \ref{UpOut} to $\gamma$, $D_2$ and $E$. Repeated applications of this lemma allow us to move the maximal $\mathcal H_0$ subpaths of $\gamma$ that lie below level $A-1$ to level $A-1$ (modulo some paths in $\Gamma-E$). Notice that if $\alpha$ is a maximal $H_0$ subpath of $\gamma$ in a lowest level $L<A-1$ of $\gamma$, then $\alpha$ has image in $\Gamma-D_2$ so that hypotheses ({\it A.}) and ({\it B.}) of Lemma \ref{UpOut} are satisfied. Apply Lemma \ref{UpOut} to all maximal $\mathcal H_0$ subpaths of $\gamma$ in level $L$. If $M$ is one of these homotopies, then $p(M)\subset [L,L+1]$ and so avoids $C$. Say $\hat \gamma$ is the resulting path. Then $\gamma$ is homotopic to $\hat \gamma$ by a homotopy that avoids $C$. The only subpaths of $\hat \gamma$ in level $L$ have image in $\Gamma-St(E)$. If $\alpha$ is a maximal $H_0$ subpath of $\hat \gamma$ in level $L+1$ then $\alpha$ is a concatenation of $H_0$ subpaths of $\gamma$ (which have image in $\Gamma-D_2$) and paths of the form $\gamma_i'$ arising from Lemma \ref{UpOut}. As the end points of these $\gamma_i'$ can be connected by edge paths in $\Gamma-D_2$ in levels $\leq L+1$, we conclude that if $\hat \alpha$ is a maximal $\mathcal H_0$ subpath of $\hat \gamma$ in level $L+1$ , then the end points of $\hat \alpha$ can be joined by a path $\beta$ in $\Gamma-D_2$ such that $p(\beta(t))\leq L+1$ for all $t$. If  $e$  and $d$ are the $t$ edges at the initial  and terminal vertex of $\hat \alpha$ then $e$ ($d$) is either an edge of $\gamma$ or is in $\Gamma-E$. In either case, $e$ and $d$ avoid $D_2$. Hence Lemma \ref{UpOut} can be applied to $(e^{-1}, \hat \alpha,d)$. Repeated applications of Lemma \ref{UpOut} shows:

\medskip

\noindent {\bf Fact 1.} {\it There is a homotopy $\gamma {\buildrel M_1\over \sim} \gamma_1$
with $p(M_1)\subset \Gamma-C$ such that for any vertex $v$ of $\gamma_1$ either  $p(v)\in [A-1,\infty)$ or $v\in \Gamma-E$, the homotopy $M_1$ fixes the edges of $\gamma$ in levels $A-1$ and above, and the only edges of $\gamma_1$ in levels $A$ and above are edges of $\gamma$ (and hence these edges avoid $D_2$). Any $t$-edge of $\gamma_1$ not in $\Gamma-E$ is also a $t$-edge of $\gamma$ and so is in $\Gamma-D_2$.  Furthermore: 

$(\ast \ast)$ Suppose $\tau$ is a maximal $\mathcal H_0$ subpath of $\gamma_1$ with $p(\tau)=A-1$. Let $v_0$ and $v_1$ be the end points $\tau$. Then Lemma \ref{UpOut} implies there is an edge path $\beta$ with $im(\beta)\subset \Gamma-D_2$ connecting $v_0$ to $v_1$ with $p(\beta)\subset (-\infty,A-1]$. If $W$ ($L$) is the component of $v\Gamma_0-D_1$ containing $v_0$ ($v_1$) then $(\ast)$ implies $W\times L$ satisfies condition $(i)$ of Lemma \ref{Like4.4}. }

\begin{figure}
\vbox to 3in{\vspace {-1.5in} \hspace {.3in}
\hspace{-1 in}
\includegraphics[scale=1]{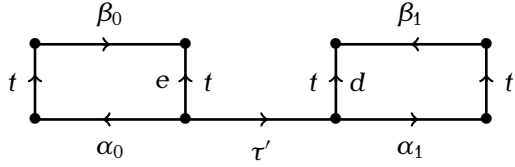}
\vss }
\vspace{-2in}
\caption{Moving $t$-Edges Horizontally}
\label{Fig7}
\end{figure}

The next step is to move all $t$ edges of $\gamma_1$ outside of $E$.  Apply Lemma \ref{push} (with $C=D_1$ and $D=D_2$) to each $t$-edge of $\gamma_1$ not already in $\Gamma-E$, to obtain a homotopy $\gamma_1{\buildrel M_2\over\sim} \gamma_2$ in $\Gamma-D_1$ where each $t$-edge of $\gamma_2$ has image in $\Gamma -E$. It is important to observe that if $e$ is a $t$ edge of $\gamma_1$ connecting levels $A-1$ and $A$, then the homotopy $M_2$ uses Lemma \ref{push} to replace $e$ by a path $(\alpha, e_1,\beta)$ in $\Gamma-D_1$ connecting the end points of $e$ such that $p(\alpha)=A-1$, $p(\beta)=A$ and $e_1$ (in $\Gamma-E$) has label $t$. Since $\alpha$ avoids $D_1$, there is an  analogue of $(\ast\ast)$ for $\gamma_2$:

\medskip

\noindent $(\overline{\ast \ast})$ {\it If $\tau$ is a maximal subpath of $\gamma_2$ in level $A-1$, then $\tau=(\alpha_0^{-1}, \tau', \alpha_1)$ where $\tau'$ is a maximal subpath of $\gamma_1$ in level $A-1$ and each of the paths $\alpha_0$ and $\alpha_1$ is either trivial (if the corresponding end point of $\tau'$ is already in $\Gamma-E$) or comes from our replacement of edges labeled $t$ at the end points of $\tau'$ (see Figure \ref{Fig7}). So $\alpha_0$ and $\alpha_1$ are in level $A-1$, avoid $D_1$ and both end in $\Gamma-E$. By $(\ast\ast)$ the end points of $\tau'$ are in unbounded components $W$ and $L$ of $v\Gamma_0-D_1$ and $W\times L$ satisfies condition $(i)$ of Lemma \ref{Like4.4}. As $\alpha_0$ and $\alpha_1$ avoid $D_1$, the end points of $\tau$ are in $W$ and $L$ as well.}

\begin{figure}
\vbox to 3in{\vspace {-1.5in} \hspace {-.1in}
\hspace{-1 in}
\includegraphics[scale=1]{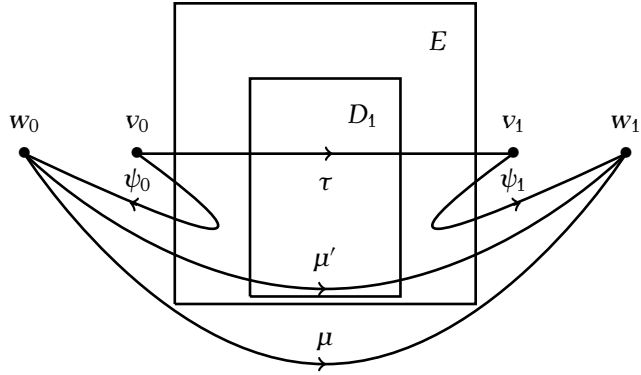}
\vss }
\vspace{-.8in}
\caption{Moving $\tau$ Down and Out}
\label{Fig8}
\end{figure}

Again let $\tau$ (with end points $v_0$ and $v_1$) be a maximal subpath of $\gamma_2$ in level $A-1$. Then $v_0,v_1\in v_0\Gamma_0-E$. If $W_i$ is the component of $v_0\Gamma_0-D_1$ containing $v_i$ (for $i\in\{0,1\}$) then ($\overline{\ast \ast}$) implies that $W_0\times W_1$ satisfies condition $(i)$ of Lemma \ref{Like4.4}. Choose $N>1$ such that $E\subset C_N$ and vertex pairs $(w_0, w_1)\in W_0\times W_1$ such that $deg(w_0,w_1)\geq N$. 
For $i\in\{0,1\}$ let $\psi_i$  be an edge path in $W_i$  from $v_i$  to $w_i$. Let $\mu$ be an edge path from $w_0$ to $w_1$ in $\Gamma-C_M$ ($\subset \Gamma-E$) such that $p(\mu)\subset (-\infty, A-1]$ (see Figure \ref{Fig8}).  Consider the edge path loop $\rho=(\tau, \psi_1, \mu^{-1}, \psi_0^{-1})$ (in levels $\leq A-1$). The loop $\rho$ is homotopic (in the image of $\rho$) to a loop $\rho'$ without backtracking. Lemma \ref{smallmove} can then be applied to $\rho'$.  After a sequence of backtracking elimination and applications of Lemma \ref{smallmove} the resulting loop 
has image in $v_0\Gamma_0$ and so is homotopically trivial in level $A-1$. In particular $\tau$ is homotopic to $(\psi_0, \mu, \psi_1^{-1})$ by a homotopy in levels $\leq A-1$ and hence in $\Gamma-C$. Since $\mu$ has image in $\Gamma-E$ and (for $i\in\{1,2\}$) $\psi_i$ is an edge path in level $A-1$ with end points in $\Gamma-E$ and image in $\Gamma-D_1$, we have shown:

\medskip 

$(\dagger)$ {\it The path $\gamma$ is homotopic $rel\{r\}$ to an edge path loop $\gamma_3$ in $\Gamma-D_1$ by a homotopy in $\Gamma-C$ such that each $t$-edge of $\gamma_3$ has image in $\Gamma-E$ and each edge of $\gamma_3$ in any level $<A-1$ is in $\Gamma-E$.} 

\medskip

Next suppose $\alpha$ is a maximal subpath of $\gamma_3$ in level $L\in [A-1,B]$. The end points of $\alpha$ are in $\Gamma-E$ and $\alpha$ has image in $\Gamma-D_1$. By Lemma \ref{slideC}, the path $\alpha$ is homotopic $rel\{0,1\}$ to an edge path in $\Gamma-E$ by a homotopy in $\Gamma-C$. Apply Lemma \ref{slideC} to all such $\alpha$. The resulting path, call it $\gamma_4$, is homotopic $rel\{0,1\}$ to $\gamma_3$ by a homotopy in $\Gamma-C$. Notice that each edge of $\gamma_4$ in a level $\leq B$ is in $\Gamma-E$. At this point we need only be concerned about maximal subpaths of $\gamma_4$ which are  in a level above level $B$ (and hence above the compact set $C$). Repeated applications of Lemma \ref{slideup} to such paths with $F=\emptyset$ (see Remark \ref{Rslideup}) move such paths out of $E$ by homotopies in $\Gamma-C$. This finishes the proof that $G$ has semistable fundamental group at $\infty$.
\end{proof}

\section{The Proof of the Main Theorem}\label{PMT}

Before proving Theorem \ref{SSDecomp} we prove Lemma \ref{SS2} (by an argument that parallels the proof of Theorem 3 of \cite{M4}).

\begin{lemma} \label{SS2} {\bf(SS2)} Suppose $G=A\ast_CB$ or $G=A\ast_C$, where $A$ and $B$ are finitely presented, $C$ is infinite, finitely generated  and in the first case, of finite index in $B$. If $A$ is 1-ended and semistable at $\infty$, then $G$ is 1-ended and semistable at $\infty$.
\end{lemma}
\begin{proof} We consider the case $G=A\ast_CB$ first. If $B=C$ or $A=C$, then $A$ has finite index in $G$ and so $G$ is 1-ended and semistable at $\infty$ (since semistability and the number of ends of a group are quasi-isometry invariants). We assume $A\ne C\ne B$. Let $\mathcal A\cup \mathcal B$ be a finite generating set for $G$, where $\mathcal A$ generates $A$, $\mathcal B$ generates $B$ and $\mathcal A\cap \mathcal B=\mathcal C$ generates $C$.  
Let $\mathcal P$ be a finite presentation for $G$ with generators $\mathcal A\cup \mathcal B$.  Let $\Gamma$ be the Cayley 2-complex for $\mathcal P$. We show $\Gamma$ has semistable fundamental group at $\infty$. Let $\Gamma_H\subset \Gamma$ be the corresponding Cayley graph of $H=A$, $B$ or $C$ with respect to $\mathcal A$, $\mathcal B$, and $\mathcal C$ respectively. Choose a geodesic edge path line $\ell=(\ldots , c_{-1}, c_0,c_1,\ldots)$ in $\Gamma_C$ (so that each edge $c_i$ has label $\bar c_i\in\mathcal C^{\pm 1}$. Assume the initial vertex of $c_0$ is $\ast$, the identity vertex of $\Gamma_C$). Let $r^+=(c_0,c_1,\ldots)$ and $r^-=(c_{-1}^{-1}, c_{-2}^{-1},\ldots)$ be the two (opposite direction) geodesic rays (beginning at $\ast$) of $\ell$. Recall that the vertices of $\Gamma$ are the elements of $G$. Given any compact set $D\subset \Gamma$ there are only finitely many vertices $v\in \Gamma$ such that both $vr^+$ and $vr^-$ intersect $D$. 

Let $s=(s_0,s_1,\ldots)$ be a proper edge path ray in $\Gamma$ with initial point $\ast$. Then the label of $s_i$ is $\bar s_i\in \mathcal A^{\pm 1}\cup\mathcal B^{\pm 1}$. Let $\ast=w_0,w_1,\ldots$ be the consecutive vertices of $s$. 
For any vertex $v\in \Gamma$ let $B(v,N)$ be the ball of radius $N$ about $v$ in the 1-skeleton of $\Gamma$. Select $r_0=r^+$, and for $i>0$, set $r_i$ equal to $w_ir^+$ or $w_ir^-$ so that for any  integer $N$ only finitely many $r_i$ intersect $B(\ast,N)$. (See Figure \ref{Fig9}.)

\begin{figure}
\vbox to 3in{\vspace {-1.5in} \hspace {-.7in}
\hspace{-1 in}
\includegraphics[scale=1]{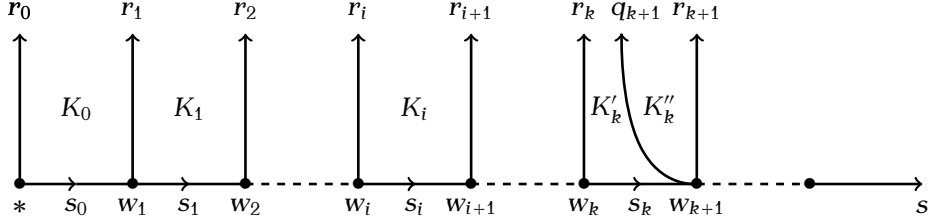}
\vss }
\vspace{-1.75in}
\caption{Stacking homotopies}
\label{Fig9}
\end{figure}

If $\bar s_{i}\in\mathcal A^{\pm 1}$ then the rays $r_i$ and $(s_{i}, r_{i+1})$ are both $\mathcal A$-rays and so are properly homotopic $rel\{w_i\}$ by a homotopy $K_i$ in $w_i\Gamma_A$ (see Theorem \ref{ssequiv}(4)). There may be infinitely many $j$ such that $w_j\Gamma_A=w_i\Gamma_A$. Let $J_i=\{j\geq 0: w_j\Gamma_A=w_i\Gamma_A\hbox{ and }\bar s_{j}\in \mathcal A^{\pm 1}\}$. Theorem \ref{ssequiv}(3) implies that the set of all $K_j$ such that $j\in J_i$, can be selected  so that only finitely many have image which intersects any given compact subset of $\Gamma$. Since only finitely many translates of $\Gamma_A$ intersect a given compact set, we may assume that: If $\bar s_{i}\in \mathcal A^{\pm1}$ then $K_i$ is defined and for any compact set $D$, only finitely many such $K_i$ have image that intersects $D$. 

If $\bar s_{k}\in \mathcal B^{\pm 1}$ then the ray $(s_{k}^{-1}, r_k)$ is a $\mathcal B$ ray. Since $C$ has finite index in $B$, there is a proper $\mathcal C$ ray $q_{k+1}$ at $w_{k+1}$ that is properly homotopic  $rel\{w_{k+1}\}$ to $(s_{k}^{-1}, r_k)$ by a homotopy $K_k'$  whose image is within $N_{\ref{SSL1}}+N_{\ref{FIclose}}$ of the image of $r_k$.  (Let $\beta_0$ be the trivial path at $w_{k+1}$. If $e_i$ is the $i^{th}$ edge of $(s_k^{-1}, r_k)$ then let $\beta_i$ be an edge path of length $\leq N_{\ref{FIclose}}$ from the terminal point of $e_i$ to a vertex of $w_{k+1}\Gamma_ C$ . Apply Lemma \ref{SSL1} to the paths $(\beta_{i-1}^{-1}, e_i, \beta_{i})$ and combine the resulting homotopies.)

Now $q_{k+1}$ and $r_{k+1}$ are both $\mathcal A$ (in fact, $\mathcal C$) rays at $w_{k+1}$ and so are properly homotopic by a homotopy $K_k''$ in $w_{k+1}\Gamma_A$. Again, Theorem \ref{ssequiv}(3) allows us to choose the $K_k''$  such that for any compact set $D$, only finitely many such $K_k''$ have image that intersect $D$. 



Combine $K_k'$ and $K_k''$ to form $K_k$, a proper homotopy of $r_k$ to $(s_{k},r_{k+1})$. Combining the $K_i$ for $i\geq 0$, provides a proper homotopy from $r^+$ to  $s$ (see Figure \ref{Fig9}). Hence every proper edge path ray at $\ast$ is properly homotopic to $r^+$ and $G$ is semistable at $\infty$ (see Theorem \ref{ssequiv}(4)). As any two proper rays in $\Gamma$ are properly homotopic, the space $\Gamma$ (and hence the group $G$) is 1-ended. 

In the second case semistability follows directly from Theorem \ref{MTComb}.
It remains to show that $A\ast_C$ is 1-ended. Let $\mathcal C\subset \mathcal A$ be finite generating sets for $C$ and $A$ respectively. Let $t$ be the stable letter and for each $c\in \mathcal C$, let $w_c$ be an $\mathcal A$ word such that $t^{-1} ct=w_c$ is a conjugation relation for $G$. Let $\Gamma$ be the Cayley graph of $G$ with respect to $\mathcal A\cup \{t\}$ and let $\Gamma_C\subset \Gamma$ be the Cayley graph of $C$ with respect to $\mathcal C$. Again choose a geodesic edge path line  $\ell=(\ldots , c_{-1}, c_0,c_1,\ldots)$ in $\Gamma_C$, so that each edge $c_i$ has label $\bar c_i\in\mathcal C^{\pm 1}$. Assume the initial vertex of $c_0$ is $\ast$, the identity vertex of $\Gamma_C$. Let $r^+=(c_0,c_1,\ldots)$ and $r^-=(c_{-1}^{-1}, c_{-2}^{-1},\ldots)$ be the two (opposite direction) geodesic rays (beginning at $\ast$) of $\ell$. Given any compact set $D\subset \Gamma$ there are only finitely many vertices $v\in \Gamma$ such that both $vr^+$ and $vr^-$ intersect $D$. Choose $E(D)$ a finite subcomplex of $\Gamma$ such that if $v$ is a vertex of $\Gamma-E$ then either $vr^+$ or $vr^-$ avoids $D$. It is enough to show that any two vertices $v$ and $w$ in $\Gamma-E$ can be joined by an edge path in $\Gamma-D$. Let $r_v$ be $vr^+$ or $vr^-$ avoiding $D$. Similarly select $r_w$ avoiding $D$. It is enough to show that $r_v$ and $r_w$ converge to the same end of $\Gamma$. Let $\alpha=(a_0,a_1,\ldots, a_n)$ be an edge path from $v$ to $w$. Assume $a_i$ has label $\bar a_i\in  \mathcal A^{\pm 1} \cup \{t^{\pm 1}\}$).  For $i\in \{1,\ldots, n\}$ let $v_i$ be the initial vertex of $a_i$ and let $r_i$ be $v_ir^+$.  Let $v_0=v$, $v_{n+1}=w$, $r_0=r_v$ and $r_{n+1}=r_w$. It is enough to show that $r_i$ and $r_{i+1}$ converges to the same end of $\Gamma$ for $i\in\{0,\ldots, n\}$. Choose $i\in \{0,1,\ldots, n\}$. If $\bar a_i\in \mathcal A^{\pm 1}$ then the proper rays $r_i$ and $(a_i, r_{i+1})$ form a proper line in the 1-ended space $v_i\Gamma_1$. Hence $r_i$ and $(a_i, r_{i+1})$ converge to the same end of $v_i\Gamma_A$. In particular, $r_i$ and $r_{i+1}$ converge to the same end of $\Gamma$. 
If $\bar a_i=t$ then write $r_i=(e_1,e_2,\ldots)$.  Let $t_k$ be the edge labeled $t$ at the initial vertex of $e_k$. There is an $\mathcal A$ path $z_j$ (labeled by one of the conjugation words $w_k^{\pm 1}$), connecting the end points of $(t_k^{-1}, e_k, t _{k+1})$. As $(z_1,z_2,\ldots)$ tracks $r_i$, it is proper and converge to the same end of $\Gamma$ as does $r_i$. Also $(z_1,z_2,\ldots)$ and $r_{i+1}$ are in the 1-ended space $v_{i+1}\Gamma_A$, so they converge to the same  end of $\Gamma$. Hence $r_i$ and $r_{i+1}$ converge to the same end of $\Gamma$. If $\bar a_i=t^{-1}$ then consider the $t$ edges at the vertices of $r_{i+1}$ and proceed as before. 
\end{proof}

\begin{proof} (of Theorem \ref{SSDecomp}) Suppose $\mathcal G$ is our graph of groups decomposition of $G$. First a reduction. Suppose there is a vertex $v$ of $\mathcal G$ such that $G_v$ is 1-ended and semistable at $\infty$. If there is an edge (loop) $e(v,v)$ of $\mathcal G$, then the HNN-extension of $G_v$ determined by $e(v,v)$ is 1-ended and semistable at $\infty$ by Lemma \ref{SS2}. In this situation, replace $v$ and $e(v,v)$ by a single vertex $v'$ with vertex group the HNN-extension of $G_v$. The resulting graph of groups is still a decomposition of $G$ and satisfies the hypothesis of our theorem. We may assume:

\medskip

\noindent (0) {\it If $v$ is a vertex of $\mathcal G$ and $G_v$ is 1-ended and semistable at $\infty$, then no edge at $v$ is a loop.}

Next, suppose there is a vertex $v$ of $\mathcal G$ such that $G_v$ is not both 1-ended and semistable at $\infty$ and an edge (loop) $e(v,v)$ (giving an HNN-extension). Then there are subgroups $G_1$ and $G_2$ of $G_v$ (both isomorphic to $G_e$) and an isomorphism $\phi:G_1\to G_2$ determining a subgroup of $G$ that is an HNN-extension of $G_v$. If either $G_1$ or $G_2$ has finite index in $G_v$, then Theorem \ref{MMFI} implies that the resulting HNN-extension (call it $G_{v'}$) of $G_v$ is 1-ended and semistable at $\infty$. In this case $v$ and the loop $e(v,v)$ are replaced by $v'$.
If $\mathcal G$ has only 1-vertex $v$, then our hypotheses and reductions imply that $\mathcal G$ cannot have an edge so that $G=G_v$. Our hypotheses imply that $G$ is 1-ended and semistable at $\infty$. Assume the result is true for $\mathcal G$ with $N\geq 1$ vertices. Now assume $\mathcal G$ has $N+1$ vertices. Suppose $e(v,w)$ is an edge of $\Gamma$, and the vertex group $G_v$  (of $\mathcal G$) is 1-ended and semistable at $\infty$. Our reduction (0)  implies that $v\ne w$. If $G_e$ has finite index in $G_w$,  then Lemma \ref{SS2} implies that the subgroup of $G$ generated by $G_v$ and $G_w$ is 1-ended and semistable at $\infty$. If $G_w$ is 1-ended and semistable at $\infty$ then Theorem \ref{MComb} implies that the subgroup of $G$ generated by $G_v$ and $G_w$ is 1-ended and semistable at $\infty$. In either case, reducing $\mathcal G$ by combining $e$ into a single vertex (with group generated by $G_v\cup G_w$) gives a graph of groups decomposition of $G$ with fewer vertices and satisfying the hypothesis of our theorem. Inductively, $G$ is 1-ended and semistable at $\infty$. So we may assume:

\medskip

\noindent (1) {\it If $v$ is a vertex of $\mathcal G$ and $G_v$ is  1-ended and semistable at $\infty$, then each edge $e(v,w)$ of $\mathcal G$ is such that $v\ne w$, $G_e$ has infinite index in $G_w$, and $G_w$ is not 1-ended and semistable at $\infty$.}

\medskip

Similarly, Theorems \ref{MMFI} and \ref{MainCM} imply that we may assume:

\medskip

\noindent (2) {\it No edge (loop) $e(v,v)$ of $\mathcal G$ is such that $G_e$ has finite index in $G_v$ and no edge $e(v,w)$ with $v\ne w$ is such that $G_e$ has finite index in both $G_v$ and $G_w$. }

\medskip





By (1), there is a vertex $v_0\in \mathcal G$ such that $G_{v_0}$ is not 1-ended and semistable at $\infty$. By our hypotheses there is an edge $e_0=e(v_0,v_1)$  of $\mathcal G$ such that $G_{e_0}$ has finite index in $G_{v_0}$. By (2), $v_0\ne v_1$ and $G_{e_0}$ has infinite index in $G_{v_1}$. By (1), $G_{v_1}$ is not 1-ended and semistable at $\infty$. By our hypothesis there is an edge $e_1=e(v_1,v_2)$  of $\mathcal G$ such that $G_{e_1}$ has finite index in $G_{v_1}$ (so $e_1\ne e_0$) and by (2), $v_1\ne v_2$ and  $G_{e_1}$ has infinite index in $G_{v_2}$. By (1), $G_{v_2}$ is not 1-ended and semistable at $\infty$.  Continuing, we arrive at a simple (non-crossing) edge path loop $\mathcal L$ in $\mathcal G$ (relabeling if necessary) $e_0,e_1,\ldots, e_n$ where $e_i=e(v_i, v_{i+1})$, $v_i\ne v_{i+1}$,  $v_{n+1}=v_0$, no $G_{v_i}$ is 1-ended and semistable at $\infty$, and $G_{e_i}$ has finite index in $G_{v_i}$ and infinite index in $G_{v_{i+1}}$.  It is enough to show the fundamental group $L$ of this loop of groups is 1-ended and semistable at $\infty$ (since then we could collapse the loop to a single vertex and the resulting graph of groups decomposition of $G$ would have fewer vertices and still satisfy the hypothesis of our theorem).

Let $A$ be the fundamental group of the graph of groups obtained by removing $e_n$ from $\mathcal L$. We consider $L$ to be an HNN-extension of $A$, with stable letter $t$ and $t^{-1}ct\in G_{v_0}$ for $c\in G_{e_n}<G_{v_{n}}$.   The group $C=G_{e_0}\cap G_{e_1}\cap \cdots \cap G_{e_n}$ has finite index in $G_{v_0}$. The subgroup of $L$ generated by $G_{v_0}$ and $t$ is an HNN extension with base $G_0$ and associated subgroup $C$ (of finite index in $G_0$). Theorem \ref{MMFI} implies this subgroup is finitely presented, 1-ended and semistable at $\infty$. Similarly each of the vertex groups of $\mathcal L$ belong to a subgroup of $L$ that is finitely presented, 1-ended and semistable at $\infty$. Repeated applications of Theorem \ref{MComb} imply that $L$ is 1-ended and semistable at $\infty$. 
\end{proof}

\bibliographystyle{amsalpha}
\bibliography{paper1}{}

\def\cprime{$'$}
\providecommand{\bysame}{\leavevmode\hbox to3em{\hrulefill}\thinspace}
\providecommand{\MR}{\relax\ifhmode\unskip\space\fi MR }
\providecommand{\MRhref}[2]{%
  \href{http://www.ams.org/mathscinet-getitem?mr=#1}{#2}
}
\providecommand{\href}[2]{#2}
\begin{thebibliography}{Mih86b}

\bibitem[BM91]{BM91}
Mladen Bestvina and Geoffrey Mess, \emph{The boundary of negatively curved
  groups}, J. Amer. Math. Soc. \textbf{4} (1991), no.~3, 469--481. \MR{1096169}

\bibitem[Bow99]{Bow99B}
B.~H. Bowditch, \emph{Connectedness properties of limit sets}, Trans. Amer.
  Math. Soc. \textbf{351} (1999), no.~9, 3673--3686. \MR{1624089}

\bibitem[Bri93]{BR93}
Stephen~G. Brick, \emph{Quasi-isometries and ends of groups}, J. Pure Appl.
  Algebra \textbf{86} (1993), no.~1, 23--33. \MR{1213151}

\bibitem[CM14]{CM2}
Gregory~R. Conner and Michael~L. Mihalik, \emph{Commensurated subgroups,
  semistability and simple connectivity at infinity}, Algebr. Geom. Topol.
  \textbf{14} (2014), no.~6, 3509--3532. \MR{3302969}

\bibitem[Geo08]{G}
Ross Geoghegan, \emph{Topological methods in group theory}, Graduate Texts in
  Mathematics, vol. 243, Springer, New York, 2008. \MR{2365352}

\bibitem[HM]{HM20}
Matthew Haulmark and Michael~L. Mihalik, \emph{Relatively hyperbolic groups
  with semistable peripheral subgroups}, Internat. J. Algebra Comput. {\bf 32}
  (2022) no. 4, 753-783. ArXiv 2101.05923.

\bibitem[Lev98]{Lev98}
Gilbert Levitt, \emph{Non-nesting actions on real trees}, Bull. London Math.
  Soc. \textbf{30} (1998), no.~1, 46--54. \MR{1479035}

\bibitem[Mih83]{M1}
Michael~L. Mihalik, \emph{Semistability at the end of a group extension},
  Trans. Amer. Math. Soc. \textbf{277} (1983), no.~1, 307--321. \MR{690054}

\bibitem[Mih85]{HNN1}
\bysame, \emph{Ends of groups with the integers as quotient}, J. Pure Appl.
  Algebra \textbf{35} (1985), no.~3, 305--320. \MR{777262}

\bibitem[Mih86a]{M86}
\bysame, \emph{Ends of double extension groups}, Topology \textbf{25} (1986),
  no.~1, 45--53. \MR{836723}

\bibitem[Mih86b]{M4}
\bysame, \emph{Semistability at {$\infty$} of finitely generated groups, and
  solvable groups}, Topology Appl. \textbf{24} (1986), no.~1-3, 259--269,
  Special volume in honor of R. H. Bing (1914--1986). \MR{872498}

\bibitem[Mih96]{M96}
\bysame, \emph{Semistability of {A}rtin and {C}oxeter groups}, J. Pure Appl.
  Algebra \textbf{111} (1996), no.~1-3, 205--211. \MR{1394352}

\bibitem[Mih16]{M6}
\bysame, \emph{Semistability and simple connectivity at {$\infty$} of finitely
  generated groups with a finite series of commensurated subgroups}, Algebr.
  Geom. Topol. \textbf{16} (2016), no.~6, 3615--3640. \MR{3584269}

\bibitem[MS21]{MS18}
M.~Mihalik and E.~Swenson, \emph{Relatively hyperbolic groups with semistable
  fundamental group at infinity}, J. Topol. \textbf{14} (2021), no.~1, 39--61.
  \MR{4186133}

\bibitem[MT92]{MT1992}
Michael~L. Mihalik and Steven~T. Tschantz, \emph{Semistability of amalgamated
  products and {HNN}-extensions}, Mem. Amer. Math. Soc. \textbf{98} (1992),
  no.~471, vi+86. \MR{1110521}

\bibitem[Rat07]{Ratt07}
Diego Rattaggi, \emph{Three amalgams with remarkable normal subgroup
  structures}, J. Pure Appl. Algebra \textbf{210} (2007), no.~2, 537--541.
  \MR{2320016}

\bibitem[She]{SSh}
Sam Shepherd, \emph{Semistability of cubulated groups}, arXiv:2203.11244
  [math.GR].

\bibitem[Swa96]{Swarup}
G.~A. Swarup, \emph{On the cut point conjecture}, Electron. Res. Announc. Amer.
  Math. Soc. \textbf{2} (1996), no.~2, 98--100. \MR{1412948}

\end{thebibliography}

\enddocument